\documentclass[12pt]{amsart}
\usepackage{times}
\usepackage{amsmath,amsfonts,amssymb,amsopn,amscd,amsthm}
\usepackage{systeme,mathbbol}
\usepackage{mathrsfs}
\usepackage{graphicx,xspace}
\usepackage{epsfig}
\usepackage{enumerate} 
\usepackage{lscape}
\usepackage{enumitem}
\usepackage{xcolor}
\usepackage[all]{xypic}

\usepackage[T1]{fontenc}
\usepackage[utf8]{inputenc}

\usepackage{hyperref}
\usepackage{psfrag}
\usepackage{bm}
\usepackage{youngtab}
\usepackage{ytableau}
\usepackage{tikz}
\usetikzlibrary{snakes}
\usepackage[all]{xy}
\usepackage{varioref}

\theoremstyle{plain}
\newtheorem{theoreme}{Theorem}[section]
\newtheorem{prop}[theoreme]{Proposition}

\newtheorem{cor}[theoreme]{Corollary}

\theoremstyle{remark}

\newtheorem{ex}[theoreme]{Example}
\newtheorem{remark}[theoreme]{Remark}
\newtheorem{c-ex}[theoreme]{Counter-example}

\date{}

\DeclareUnicodeCharacter{03C7}{\chi}
\DeclareUnicodeCharacter{B2}{^{2}}

\DeclareMathOperator{\ty}{type}
\DeclareMathOperator{\mty}{marked-type}
\DeclareMathOperator{\SC}{SC}
\DeclareMathOperator{\DB}{DB}

\DeclareMathOperator{\Res}{Res}

\DeclareMathOperator{\diag}{diag}
\DeclareMathOperator{\Id}{Id}

\DeclareMathOperator{\character}{char}
\DeclareMathOperator{\height}{ht}

\DeclareMathOperator{\Gather}{Gather}

\DeclareMathOperator{\blambda}{\bm\lambda}
\DeclareMathOperator{\bnu}{\bm\nu}
\DeclareMathOperator{\bmu}{\bm\mu}
\DeclareMathOperator{\brho}{\bm\rho}
\DeclareMathOperator{\bdelta}{\bm\delta}

\DeclareMathOperator{\bsigma}{\bm\sigma}

\definecolor{ududff}{rgb}{0.30196078431372547,0.30196078431372547,1}

\author[O. Tout]{Omar Tout}
\address{Department of Mathematics, College of Science, Sultan Qaboos University, P. O. Box 36, Al Khod 123, Sultanate of Oman}
\email{o.tout@squ.edu.om}

\title[Generalized characters of the generalized symmetric group]{Generalized characters of the generalized\\ symmetric group}

\makeatletter
\@namedef{subjclassname@2020}{\textup{2020} Mathematics Subject Classification}
\makeatother
\subjclass[2020]{05E10, 20C30.}
\keywords{Gelfand pairs, symmetric pairs, generalized symmetric groups, generalized characters, Murnaghan-Nakayama rule}
 
\begin{document}
\maketitle

\begin{abstract} We prove that $(\mathbb{Z}_k \wr \mathcal{S}_n \times \mathbb{Z}_k \wr \mathcal{S}_{n-1}, \diag (\mathbb{Z}_k \wr \mathcal{S}_{n-1}) )$ is a symmetric Gelfand pair, where $\mathbb{Z}_k \wr \mathcal{S}_n$ is the wreath product of the cyclic group $\mathbb{Z}_k$ with the symmetric group $\mathcal{S}_n.$ The proof is based on the study of the $\mathbb{Z}_k \wr \mathcal{S}_{n-1}$-conjugacy classes of $\mathbb{Z}_k \wr \mathcal{S}_n.$ We define the generalized characters of $\mathbb{Z}_k \wr \mathcal{S}_n$ using the zonal spherical functions of $(\mathbb{Z}_k \wr \mathcal{S}_n \times \mathbb{Z}_k \wr \mathcal{S}_{n-1}, \diag (\mathbb{Z}_k \wr \mathcal{S}_{n-1}) ).$ We show that these generalized characters have properties similar to usual characters. A Murnaghan-Nakayama rule for the generalized characters of the hyperoctahedral group is presented. The generalized characters of the symmetric group were first studied by Strahov in \cite{strahov2007generalized}. 
\end{abstract}

\section{Introduction}

Let $K$ be a subgroup of a finite group $G.$ The pair $(G,K)$ is called a \textit{Gelfand pair} if its associated double-class algebra $\mathbb{C}[K\backslash G/K]$ is commutative. It turns out that $(G,K)$ is a Gelfand pair if and only if the induced representation $1_K^G$ of $K$ on $G$ is multiplicity free, that is each irreducible representation of $G$ appears at most once in the  decomposition of $1_K^G$ as sum of irreducible representations. If $(G,K)$ is a Gelfand pair, the algebra of constant functions over the $K$-double-classes in $G$ has a particular basis, whose elements are called \textit{zonal spherical functions} of the pair $(G,K).$ For example, the pair $(G\times G,\diag(G)),$ where $\diag(G)$ is the diagonal subgroup of $G\times G,$ is a Gelfand pair for any finite group $G,$ see \cite[Section VII.1]{McDo}. Its zonal sphercial functions are the normalized irreducible characters of $G.$\\

The pair $(G,K)$ is said to be \textit{symmetric} if the action of $G$ on $G/K$ is symmetric, that is if $(x,y)\in \mathcal{O}$ then $(y,x)\in \mathcal{O}$ for all orbits $\mathcal{O}$ of $G$ on $G/K\times G/K.$ If the pair $(G,K)$ is symmetric and Gelfand then it is called a \textit{symmetric Gelfand pair}. It is known, see \cite{ceccherini2010representation}, that $(G,K)$ is a symmetric Gelfand pair if and only if $g^{-1}\in KgK$ for all $g\in G.$ Many symmetric Gelfand pairs involving the symmetric group $\mathcal{S}_n$ were studied in the litterature. The pair $(\mathcal{S}_{2n}, \mathcal{H}_n),$ where $\mathcal{H}_n$ is the Hyperoctahedral subgroup of the symmetric group $\mathcal{S}_{2n},$ is a symmetric Gelfand pair studied in \cite[Section VII.2]{McDo} and \cite{toutejc}. In \cite[Example 1.4.10]{ceccherini2010representation},  the authors show that $(\mathcal{S}_{n},\mathcal{S}_{n-k}\times \mathcal{S}_{k})$ is a symmetric Gelfand pair. \\

Two elements $x,y\in G$ are said to be $K$-conjugate if there exists $k\in K$ such that $x=kyk^{-1}.$ Denote by $\diag (K):=\lbrace (k,k)\in G\times K \rbrace$ the diagonal subgroup of $G\times K.$ The pair $(G\times K,\diag (K))$ is a symmetric Gelfand pair if and only if $g^{-1}$ and $g$ are $K$-conjugate for any $g\in G.$ The symmetric Gelfand pair $(\mathcal{S}_n\times \mathcal{S}_{n-1},\diag(\mathcal{S}_{n-1}))$ was considered by Brender in \cite{brender1976spherical}. An explicit formula for its zonal spherical functions was used by Strahov in \cite{strahov2007generalized} to define the generalized characters of the symmetric group. Combinatorial properties and objects arise from these latter such as a Murnaghan–Nakayama type rule and the definition of generalized Schur functions. In \cite{tout2021symmetric}, the author shows that $(\mathcal{H}_n\times \mathcal{H}_{n-1},\diag(\mathcal{H}_{n-1}))$ is a symmetric Gelfand pair.\\

The wreath product $\mathbb{Z}_k \wr \mathcal{S}_n$ of the cyclic group $\mathbb{Z}_k$ with the symmetric group $\mathcal{S}_n$ is usually called the \textit{generalized symmetric group}. In this paper we prove that the $\mathbb{Z}_k \wr \mathcal{S}_{n-1}$-conjugacy classes of $\mathbb{Z}_k \wr \mathcal{S}_n$ can be indexed by marked $k$-partite partitions of $n.$ As a consequence, any element of $\mathbb{Z}_k \wr \mathcal{S}_n$ is $\mathbb{Z}_k \wr \mathcal{S}_{n-1}$-conjugate to its inverse and $(\mathbb{Z}_k \wr \mathcal{S}_n \times \mathbb{Z}_k \wr \mathcal{S}_{n-1}, \diag (\mathbb{Z}_k \wr \mathcal{S}_{n-1}) )$ is a symmetric Gelfand pair. This result can be considered as a generalization of the fact that $(\mathcal{S}_n\times \mathcal{S}_{n-1},\diag(\mathcal{S}_{n-1}))$ is a symmetric Gelfand pair. We define the generalized characters of the generalized symmetric group to be a generalization of the generalized characters of the symmetric group studied by Strahov in \cite{strahov2007generalized}. A special attention is made for the case when $k=2,$ the hyperoctahedral group.  We give a Murnaghan-Nakayama rule and we show the entries of some generalized characters tables.

\section{The theory of Gelfand pairs}\label{sec:symm_gel_pairs}

In this section we suppose that $G$ is a finite group and $K$ is a subgroup of $G.$ We review the needed results from the theory of Gelfand pairs. Proofs for the  results stated here can be found in the book \cite{ceccherini2010representation} of T. Ceccherini-Silberstein, F. Scarabotti, and F. Tolli or in the book \cite{McDo} of I.G. Macdonald. 

\subsection{Gelfand pairs} We denote by $L(G)$ the algebra of complex-valued functions on $G$ and by $L(K\backslash G/K)$ the subalgebra of $L(G)$ of all the functions $f$ that are constant on the double-classes $KxK$ of $G,$ that is:
$$f\in L(K\backslash G/K) \text{ if and only if } f(kxk^{\prime})=f(x) \text{ for any $x\in G$ and $k,k^{\prime}\in K$}.$$
The product of two functions $f_1$ and $f_2$ in $L(G)$ is the convolution product $\ast$ defined by:
$$(f_1\ast f_2)(g):=\sum_{h\in G}f_1(gh^{-1})f_2(h) \text{ for any $g\in G.$}$$ 
The pair $(G,K)$ is called a \textit{Gelfand pair} if the algebra $L(K\backslash G/K)$ is commutative. 

Let $\hat{G}$ be a fixed set of irreducible pairwise inequivalent representations (over $\mathbb{C}$) of $G.$ A representation of $G$ is said to be \textit{multiplicity free} if each irreducible representation of $\hat{G}$ appears at most once in its decomposition as sum of irreducible representations. Consider now the algebra $\mathbb{C}[G/K]$ spanned by the set of left cosets of $K$ in $G$ and let $G$ acts on it as follows:
$$g(k_iK)=(gk_i)K \text{ for every $g\in G$ and every representative $k_i$ of $G/K$}.$$
This representation of $G$ is usually denoted $1_K^G$ and called the induced representation of $K$ on $G.$ It turns out, see \cite[$(1.1)$ page $389$]{McDo}, that $(G,K)$ is a Gelfand pair if and only if the induced representation $1_K^G$ is multiplicity free.

\subsection{Zonal spherical functions}
Suppose that $(G,K)$ is a Gelfand pair and that :
$$1_K^G=\bigoplus_{i=1}^{s}X_i,$$
where $X_i$ are irreducible representations of $G.$ We define the functions $\omega_i:G\rightarrow \mathbb{C},$ using the irreducible characters $\mathcal{X}_i,$ as follows :
\begin{equation}
\omega_i(x)=\frac{1}{|K|}\sum_{k\in K}\mathcal{X}_i(x^{-1}k) \text{ for every $x\in G.$} 
\end{equation} 
The functions $(\omega_i)_{1\leq i\leq s}$ are called the \textit{zonal spherical functions} of the pair $(G,K).$ They have a long list of important properties, see \cite[page 389]{McDo}. They form an orthogonal basis for $L(K\backslash G/K)$ and they satisfy the following equality :
\begin{equation}\label{prop_fon_zon}
\omega_i(x)\omega_i(y)=\frac{1}{|K|}\sum_{k\in K}\omega_i(xky),
\end{equation}
for every $1\leq i\leq s$ and for every $x,y\in G.$

\subsection{Symmetric Gelfand pairs} Suppose now that the group $G$ acts on a finite set $X.$ We say that an orbit $\mathcal{O}$ of $G$ on $X\times X$ is symmetric if $(x,y)\in \mathcal{O}$ implies that $(y,x)\in \mathcal{O}.$ In addition, if all orbits of $G$ on $X\times X$ are symmetric we say that the action of $G$ on $X$ is symmetric. The pair $(G,K)$ is called \textit{symmetric} if the action of $G$ on $G/K$ is symmetric. If in addition $(G,K)$ is a Gelfand pair then we say that $(G,K)$ is a symmetric Gelfand pair. It can be shown, see \cite[Exercise 1.5.25]{ceccherini2010representation} that $(G,K)$ is a symmetric Gelfand pair if and only if $g^{-1}\in KgK$ for all $g\in G.$

We denote by $\diag (K)$ the diagonal subgroup of $G\times K,$ that is the subgroup of $G\times K$ formed of all the elements $(k,k)$ with $k\in K:$
$$\diag (K):=\lbrace (k,k)\in G\times K \rbrace.$$
Two elements $x,y\in G$ are said to be $K$-conjugate if there exists $k\in K$ such that $x=kyk^{-1}.$

Let us denote by $C(G,K)$ the subalgebra of $L(G)$ of functions which are constant on the $K$-conjugacy classes, that is:
\begin{equation*}
C(G,K):=\lbrace f\in L(G) \text{ such that $f(x)=f(y)$ for any two $K$-conjugate elements $x,y\in G$}\rbrace.
\end{equation*}
A proof for the following proposition can be found in \cite[Page $64$]{ceccherini2010representation}.

\begin{prop} \label{prop_couple_sym_Gel}
$(G\times K,\diag (K))$ is a symmetric Gelfand pair if and only if $g^{-1}$ and $g$ are $K$-conjugate for any $g\in G.$ Moreover, if this is the case, then $C(G, K)$ is
commutative and $K$ is a multiplicity-free subgroup of $G.$
\end{prop}

\subsection{$K$-generalized characters}
Recall that any irreducible representation of $G\times K$ is of the form $X\times Y$ where $X$ and $Y$ are irreducible representations of $G$ and $K$ respectively. Suppose that $(G\times K,\diag (K))$ is a symmetric Gelfand pair with 
$$1_{\diag (K)}^{G\times K}=\bigoplus_{1\leq i\leq r \atop{1\leq j\leq s}} X_i\times Y_j, $$
where $X_i$ and $Y_j$ are irreducible representations of $G$ and $K$ respectively for any $i$ and $j.$
The zonal spherical functions of the pair $(G\times K,\diag (K))$ are then given by:
\begin{eqnarray*}
\omega_{i,j}(x,y)&=&\frac{1}{|K|}\sum_{h\in K}\mathcal{X}_i(x^{-1}h)\mathcal{Y}_j(y^{-1}h)
\end{eqnarray*} 
for any $(x,y)\in G\times K.$ In \cite[Lemma 2.1.1]{ceccherini2010representation}, the authors showed that the map $$\Phi:L(\diag(K) \backslash G\times K / \diag(K))\rightarrow C(G,K)$$ defined by:
\begin{equation*}
[\Phi(F)](g)=|K|F(g,1_G)
\end{equation*}
is a linear isomorphism of algebras and that 
\begin{equation*}
||\Phi(F)||_{L(G)}^2=|K| ||F||_{L(G\times K)}^2.
\end{equation*}
As stated before, the zonal spherical functions $(\omega_{i,j})_{1\leq i\leq r \atop{1\leq j\leq s}}$ of the pair $(G\times K,\diag (K))$ form an orthogonal basis for $L(\diag(K) \backslash G\times K / \diag(K)).$ The image of the functions $\frac{\mathcal{Y}_j(1)}{|K|}\omega_{i,j}$ by $\Phi,$ which shall be denoted $\varpi_{i,j},$ form an orthogonal basis for the algebra $C(G,K).$ They will be called the \textit{$K$-generalized characters} of the group $G$ and they are explicitly defined by  
\begin{equation}\label{Def_Gen_char}
\varpi_{i,j}(g)=\mathcal{Y}_j(1)\omega_{i,j}(g,1) \text{ for any $g\in G.$}
\end{equation}

\section{Partitions and Young diagrams}

A \textit{partition} $\lambda$ is a weakly decreasing list of positive integers $(\lambda^1,\ldots,\lambda^l)$ where $\lambda^1\geq \lambda^2\geq\ldots \geq\lambda^l\geq 1.$ The $\lambda^i$ are called the \textit{parts} of $\lambda$; the \textit{size} of $\lambda$, denoted by $|\lambda|$, is the sum of all of its parts. If $|\lambda|=n$, we say that $\lambda$ is a partition of $n.$ The set of all partitions of $n$ will be denoted $\mathcal{P}_n$ and $\mathcal{P}_0$ will be considered to be the empty set $\emptyset.$ If $j$ is a part of a partition $\lambda$ of $n$ then we will write $\lambda-(j)$ to designate the partition of $n-j$ obtained by removing the part $j$ of $\lambda.$ Similarly, we define $\lambda+(j)$ to designate the partition of $n+j$ obtained by adding the part $j$ to $\lambda.$ We will also use the exponential notation $\lambda=(1^{m_1(\lambda)},2^{m_2(\lambda)},3^{m_3(\lambda)},\ldots),$ where $m_i(\lambda)$ is the number of parts equal to $i$ in the partition $\lambda.$ In case there is no confusion, we will omit $\lambda$ from $m_i(\lambda)$ to simplify our notation. If $\lambda=(1^{m_1(\lambda)},2^{m_2(\lambda)},3^{m_3(\lambda)},\ldots,n^{m_n(\lambda)})$ is a partition of $n$ then $\sum_{i=1}^n im_i(\lambda)=n.$ We will dismiss $i^{m_i(\lambda)}$ from $\lambda$ when $m_i(\lambda)=0,$ for example, we will write $\lambda=(1,3,5^2)$ instead of $\lambda=(1,2^0,3,4^0,5^2).$

Any partition $\lambda=(\lambda^1,\ldots,\lambda^l)$ of $n$ can be represented by a Young diagram. This is an array of $n$ squares having $l$ left-justified rows with row $i$ containing $\lambda^i$ squares for $1\leq i\leq l.$ For example, the following is the Young diagram of the partition $\lambda=(5,4,2,1,1)$ of $13$
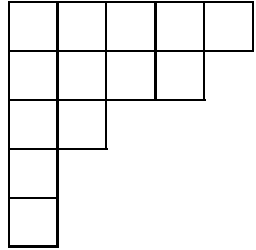
\[ \ydiagram{5,4,2,1,1} \]
If each cell of the Young diagram of a partition $\delta$ is also a cell of the Young diagram of a partition $\lambda,$ then we say that $\delta$ is included in $\lambda$ and we write $\delta\subseteq \lambda.$ In case $\delta\subseteq \lambda,$ we define the \textit{difference} $\lambda/\delta,$ usually called the \textit{skew Young diagram}, to be the set of cells in the Young diagram of $\lambda$ which are not in the Young diagram of $\delta.$ A skew Young diagram $\lambda/\delta$ is a \textit{border strip} if it is connected and does not contain any $2\times 2$ block of cells. The \textit{height}, denoted $\height,$ of a border strip is one less than the number of rows it occupies. For example, if $\lambda=(5,4,2,1,1),$ $\delta=(4,2,1),$ $\mu=(2,1)$ and $\rho=(3,1),$ then $\lambda/\delta$ and $\lambda/\mu$ are not border strips but $\lambda/\rho$ is a border strip of height $4.$ Their corresponding skew Young diagrams are respectively shown below 
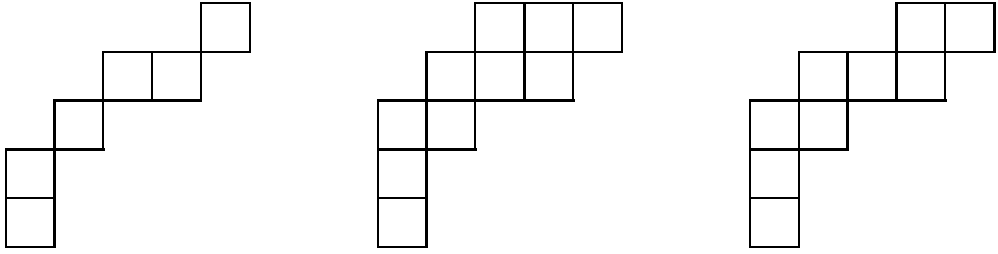
\[ \ydiagram{4+1,2+2,1+1,1,1} \,\,\,\,\,\,\,\,\,\,\,\,\,\,\,\,\,\,\,\,\,\,\,\, \ydiagram{2+3,1+3,2,1,1} \,\,\,\,\,\,\,\,\,\,\,\,\,\,\,\,\,\,\,\,\,\,\,\, \ydiagram{3+2,1+3,2,1,1}\]

A \textit{$k$-partite partition} $\blambda$ of $n$ is a $k$-tuple $(\lambda_0,\lambda_1,\ldots, \lambda_{k-1}),$ where each $\lambda_i$ is a partition and $|\lambda_0|+|\lambda_1|+\cdots +|\lambda_{k-1}|=n.$ If $\blambda$ is a $k$-partite partition and $j$ is a part of $\lambda_i$ for some $0\leq i\leq k-1,$ then we say that $j$ is a part of $\blambda$ and define $P(j)$ to be the smallest $i$ such that $j\in \lambda_i.$  All the above definitions can be naturally extended for $k$-partite partitions of $n.$ For example, we write $\bdelta\subseteq \blambda$ if $\delta_i\subseteq \lambda_i$ for any $0\leq i\leq k-1$ and in this case $\blambda/\bdelta=(\lambda_0/ \delta_0,\lambda_1/ \delta_1,\cdots ,\lambda_{k-1}/ \delta_{k-1}).$ We say that $\blambda/\bdelta$ is a border strip if one of its constituents is a border strip and the others are $\emptyset.$ We abuse notation and use $P(\blambda/\bdelta)$ to be the index of the border strip constituent and $\height(\blambda/\bdelta)$ to be its height. 

An \textit{exterior corner} of a Young diagram $\mathcal{Y}$ having $n$ squares is an extremity of a row where a new square can be added to obtain a new Young diagram with $n+1$ squares. Below we mark by bullets the four exterior corners of the Young diagram of $\lambda=(5,3,3,1,1)$
$$\young({}{}{}{}{}\bullet,{}{}{}{},{}{}{}{},{}{},{}{})\,\,\,\,\,\,\,\,\,\,\,\,\,\, \young({}{}{}{}{}{},{}{}{}\bullet,{}{}{}{},{}{},{}{})\,\,\,\,\,\,\,\,\,\,\,\,\,\, \young({}{}{}{}{}{},{}{}{}{},{}{}{}{},{}\bullet,{}{})\,\,\,\,\,\,\,\,\,\,\,\,\,\, \young({}{}{}{}{}{},{}{}{}{},{}{}{}{},{}{},{}{},\bullet)$$
If $\lambda$ is a partition of $n$ that can be obtained from the partition $\mu$ of $n-1$ by adding only one square in an exterior corner of the Young diagram of $\mu$ then we will write $\mu\nearrow \lambda.$ We extend this notation to $k$-partite partitions of $n$ and we write $\bmu\nearrow \blambda$ if $\mu_i\nearrow \lambda_i$ for some $i$ and $\mu_j= \lambda_j$ for $j\neq i.$

We will need to track information from a particular part of a partition $\lambda$ of $n.$ For this, we define a \textit{marked partition} of $n$ to be a couple $(j,\rho)$ where $1\leq j\leq n$ and $\rho\in \mathcal{P}_{n-j}.$ A \textit{marked partition} $(j,\rho)$ of $n$ can be represented by  the Young diagram of $\rho+(j)$ with the row corresponding to $j$ marked by red. Alternatively, we may mark with $\ast$ the part in a partition to designate a marked partition. For example, the marked partition $(3,(4,3,2,1))$ can be written $(4,3,3^\ast,2,1)$ and its diagram representation is as follows

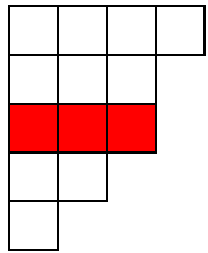
\[
\ydiagram[*(white) ]
{4,3,0,2,1}
*[*(red)]{4,3,3,2,1}
\]

If $\mu\nearrow \lambda$ where $\lambda$ is a partition of $n$ then $(j,\lambda-(j))$ is a marked partition of $n$ where $j$ is the part representing the row in which the exterior corner belongs. 

A $k$-partite partition of $n$ is said to be marked if one and only one of its constituents is a marked partition. We will write $\blambda^{(j,\lambda_i)}$ to say that $\blambda$ is marked at its $i^{th}$ constituent and the marked part is $j.$

\section{Generalised symmetric group} 

\subsection{Wreath products} The \textit{wreath product} $G \wr \mathcal{S}_n$ is the group with underlying set $G^n\times \mathcal{S}_n$ and product defined as follows:
$$((\sigma_1,\ldots ,\sigma_n); p)\cdot ((\epsilon_1,\ldots ,\epsilon_n); q)=((\sigma_{q^{-1}(1)}\epsilon_1,\ldots ,\sigma_{q^{-1}(1)}\epsilon_n);pq),$$
for any $((\sigma_1,\ldots ,\sigma_n); p),((\epsilon_1,\ldots ,\epsilon_n); q)\in G^n\times \mathcal{S}_n.$ We apply $p$ before $q$ when we write the product $pq.$ The identity in this group is $(1;1):=((1_G,1_G,\ldots ,1_G); \Id_n).$ The inverse of an element $((\sigma_1,\sigma_2,\ldots ,\sigma_n); p)\in G \wr \mathcal{S}_n$ is given by $$((\sigma_1,\sigma_2,\ldots ,\sigma_n); p)^{-1}=((\sigma^{-1}_{p(1)},\sigma^{-1}_{p(2)},\ldots ,\sigma^{-1}_{p(n)}); p^{-1}).$$

\subsection{Conjugacy classes of $G \wr \mathcal{S}_n$} Let $x = (g; p)\in G \wr \mathcal{S}_n,$ where $g = (g_1,\ldots , g_n) \in G^n$ and $p \in \mathcal{S}_n$ is written as a product of disjoint cycles. If $(i_1,i_2,\ldots, i_r)$ is a cycle of $p,$ the element $g_{i_r}g_{i_{r-1}}\ldots g_{i_1}\in G$ is determined up to conjugacy in $G$ by $g$ and $(i_1,i_2,\ldots, i_r),$ and is called the \textit{cycle product} of $x$ corresponding to the cycle $(i_1,i_2,\ldots, i_r),$ see \cite[Page $170$]{McDo}. We denote by $G_\star$ the set of conjugacy classes of $G.$ For any conjugacy class $c\in G_\star,$ we denote by $\lambda(c)$ the partition written in the exponential way where $m_i(\lambda(c))$ is the number of cycles of length $i$ in $p$ whose cycle-product lies in $c$ for each integer $i\geq 1.$ Then each element $x=(g; p)\in G \wr \mathcal{S}_n$ gives rise to a family of partitions $\Lambda=(\lambda(c))_{c\in G_\star}$ indexed by $G_\star$ 
such that 
$$\sum_{i\geq 1,c\in G_\star}im_i(\lambda(c))=n.$$ 
This family of partitions is called the \textit{type} of $x$ and denoted $\ty(x).$ It turns out, see \cite[Page $170$]{McDo}, that two permutations are conjugate in $G\wr \mathcal{S}_n$ if and only if they have the same type. Thus the conjugacy classes of $G\wr \mathcal{S}_n$ are in correspondence with families of partitions indexed by $G_\star.$ If $\Lambda=(\lambda(c))_{c\in G_\star},$ we will denote by $C_\Lambda$ its associated conjugacy class:
\begin{equation*}
C_\Lambda:=\lbrace x\in G\wr \mathcal{S}_n; \ty(x)=\Lambda\rbrace.
\end{equation*}
For a partition $\rho$ of $n,$ let 
\[
z_\rho:=\prod_{1\leq i\leq n}i^{{m_i(\rho)}}{m_i(\rho)}!.
\]
By \cite[Appendix B]{McDo}, the order of the centralizer of an element of type $\Lambda$ in $G\wr \mathcal{S}_n$ is 
\begin{equation*}
Z_\Lambda=\prod_{c\in G_\star}z_{\lambda(c)}\xi_c^{l(\lambda(c))},
\end{equation*}
where $\xi_c=\frac{|G|}{|c|}$ is the order of the centralizer of an element $g\in c$ in $G.$ Thus, if $\Lambda=(\lambda(c))_{c\in G_\star},$ the cardinal of $C_\Lambda$ is given by:
\begin{equation*}
|C_\Lambda|=\frac{|G\wr \mathcal{S}_n|}{Z_\Lambda}=\frac{|G|^nn!}{\prod\limits_{c\in G_\star}z_{\lambda(c)}\xi_c^{l(\lambda(c))}}.
\end{equation*} 

\subsection{Generalized symmetric group} The \textit{generalized symmetric group} is the wreath product $\mathbb{Z}_k \wr \mathcal{S}_n,$ where $\mathbb{Z}_k:=\lbrace 0,1,\cdots , k-1\rbrace$ is the cyclic group of order $k.$ Since there are exactly $k$ conjugacy classes of $\mathbb{Z}_k$ consisting of $\lbrace i\rbrace$ for each $0\leq i\leq k-1,$ the type of $x= (g; p)\in \mathbb{Z}_k \wr \mathcal{S}_n$ is the $k$-partite partition $\blambda=(\lambda_0,\lambda_1,\ldots, \lambda_{k-1})$ of $n,$ where each partition $\lambda_i$ is formed out of cycles $c$ of $p$ whose cycle product equals $i.$ Since $\mathbb{Z}_k$ is a group for addition, we will use \textit{cycle sum} from now on. For the remaining of this paper, it would be very important to track the cycle of $p$ containing $n.$ As such, we define the \textit{marked-type} of $x$ to be the marked $k$-partite partition of $n$ obtained from the type of $x$ by marking the part corresponding to the cycle containing $n$ in $p.$ For example, consider the element $x=(g,p)\in \mathbb{Z}_{3}\wr \mathcal{S}_{10}$ where $g=(1,1,2,0,1,1,1,2,1,0)$ and $p=(1,4)(2,5)(3)(6)(7,8,9,10).$ The cycle sum of $(1,4)$ is $1+0=1,$ of $(2,5)$ is $1+1=2,$ of $(3)$ is $2,$ of $(6)$ is $1$ and of $(7,8,9,10)$ is $1+2+1+0=1$ in $\mathbb{Z}_{3}.$ Thus $\ty(x)=(\emptyset,(4,2,1),(2,1))$ and $\mty(x)=(\emptyset,(4^\ast,2,1),(2,1)).$\\

An element $(g; p)\in \mathbb{Z}_k \wr \mathcal{S}_n$ may be represented by a diagram obtained by drawing the two lines permutation diagram associated to $p,$ with the nodes of the bottom row labeled by the elements $g_i,$ $1\leq i\leq n.$
\begin{ex}\label{example1} The element $x=(g,p)\in \mathbb{Z}_{3}\wr \mathcal{S}_{10}$ where $g=(1,1,2,0,1,1,1,2,1,0)$ and $p=(1,4)(2,5)(3)(6)(7,8,9)(10)$ is represented by the following diagram
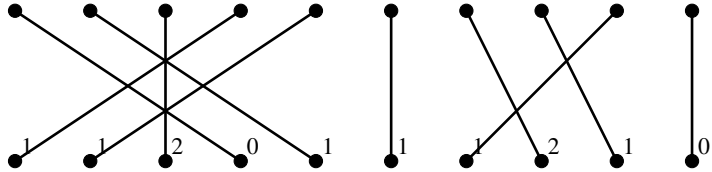
\[
\begin{tikzpicture}[line cap=round,line join=round,x=1cm,y=1cm]
\draw [line width=1pt] (-15,7)-- (-12,5);
\draw [line width=1pt] (-14,7)-- (-11,5);
\draw [line width=1pt] (-13,7)-- (-13,5);
\draw [line width=1pt] (-10,7)-- (-10,5);
\draw [line width=1pt] (-12,7)-- (-15,5);
\draw [line width=1pt] (-11,7)-- (-14,5);
\draw [line width=1pt] (-9,7)-- (-8,5);
\draw [line width=1pt] (-8,7)-- (-7,5);
\draw [line width=1pt] (-7,7)-- (-9,5);
\draw [line width=1pt] (-6,7)-- (-6,5);
\begin{scriptsize}
\draw [fill=black] (-15,7) circle (2.5pt);
\draw [fill=black] (-12,5) circle (2.5pt);
\draw[color=black] (-11.84,5.2) node {0};
\draw [fill=black] (-14,7) circle (2.5pt);
\draw [fill=black] (-11,5) circle (2.5pt);
\draw[color=black] (-10.84,5.2) node {1};
\draw [fill=black] (-13,7) circle (2.5pt);
\draw [fill=black] (-13,5) circle (2.5pt);
\draw[color=black] (-12.84,5.2) node {2};
\draw [fill=black] (-10,7) circle (2.5pt);
\draw [fill=black] (-10,5) circle (2.5pt);
\draw[color=black] (-9.84,5.2) node {1};
\draw [fill=black] (-12,7) circle (2.5pt);
\draw [fill=black] (-15,5) circle (2.5pt);
\draw[color=black] (-14.84,5.2) node {1};
\draw [fill=black] (-11,7) circle (2.5pt);
\draw [fill=black] (-14,5) circle (2.5pt);
\draw[color=black] (-13.84,5.2) node {1};
\draw [fill=black] (-9,7) circle (2.5pt);
\draw [fill=black] (-8,5) circle (2.5pt);
\draw[color=black] (-7.84,5.2) node {2};
\draw [fill=black] (-8,7) circle (2.5pt);
\draw [fill=black] (-7,5) circle (2.5pt);
\draw[color=black] (-6.84,5.2) node {1};
\draw [fill=black] (-7,7) circle (2.5pt);
\draw [fill=black] (-6,5) circle (2.5pt);
\draw[color=black] (-5.84,5.2) node {0};
\draw [fill=black] (-6,7) circle (2.5pt);
\draw [fill=black] (-9,5) circle (2.5pt);
\draw[color=black] (-8.84,5.2) node {1};
\end{scriptsize}
\end{tikzpicture}
\]
\end{ex}
The inverse of an element $(g; p)\in \mathbb{Z}_k \wr \mathcal{S}_n$ is $((-g_{p(1)},-g_{p(2)},\ldots ,-g_{p(n)}); p^{-1})$ while the product of two elements $(g; p), (h;q)\in \mathbb{Z}_k \wr \mathcal{S}_n$ can be obtained easily by drawing the diagram of $(h;q)$ below the diagram of $(g; p)$ then considering the resulted diagram. 
\begin{ex}\label{example2} Consider $x$ to be the element of the previous example. If $y=(h,q)\in \mathbb{Z}_{3}\wr \mathcal{S}_{10}$ where $h=(0,1,1,0,1,2,1,0,0,1)$ and $q=(1,3,7)(2)(10,9,4,8,5,6)$ then $$y^{-1}=((2,2,2,0,1,2,0,2,0,0); (1,7,3)(2)(10,6,5,8,4,9))$$ and $xy=((1,2,2,1,0,0,0,0,0,2);(1,8,4,3,7,5,2,6,10,9))$ as can be obtained from the following diagram
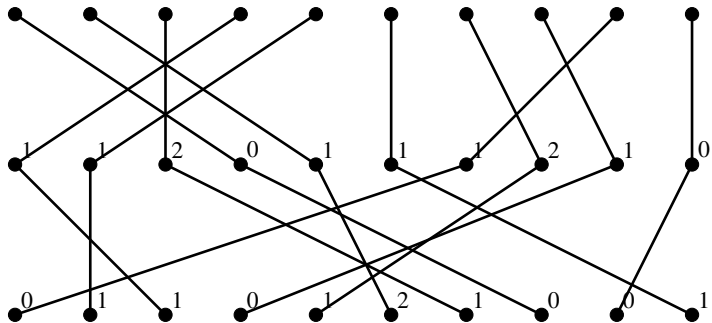
\[
\begin{tikzpicture}[line cap=round,line join=round,x=1cm,y=1cm]
\draw [line width=1pt] (-15,7)-- (-12,5);
\draw [line width=1pt] (-14,7)-- (-11,5);
\draw [line width=1pt] (-13,7)-- (-13,5);
\draw [line width=1pt] (-10,7)-- (-10,5);
\draw [line width=1pt] (-12,7)-- (-15,5);
\draw [line width=1pt] (-11,7)-- (-14,5);
\draw [line width=1pt] (-9,7)-- (-8,5);
\draw [line width=1pt] (-8,7)-- (-7,5);
\draw [line width=1pt] (-7,7)-- (-9,5);
\draw [line width=1pt] (-6,7)-- (-6,5);
\draw [line width=1pt] (-15,5)-- (-13,3);
\draw [line width=1pt] (-6,5)-- (-7,3);
\draw [line width=1pt] (-9,5)-- (-15,3);
\draw [line width=1pt] (-14,5)-- (-14,3);
\draw [line width=1pt] (-13,5)-- (-9,3);
\draw [line width=1pt] (-7,5)-- (-12,3);
\draw [line width=1pt] (-12,5)-- (-8,3);
\draw [line width=1pt] (-8,5)-- (-11,3);
\draw [line width=1pt] (-11,5)-- (-10,3);
\draw [line width=1pt] (-10,5)-- (-6,3);
\begin{scriptsize}
\draw [fill=black] (-15,7) circle (2.5pt);
\draw [fill=black] (-12,5) circle (2.5pt);
\draw[color=black] (-11.84,5.2) node {0};
\draw [fill=black] (-14,7) circle (2.5pt);
\draw [fill=black] (-11,5) circle (2.5pt);
\draw[color=black] (-10.84,5.2) node {1};
\draw [fill=black] (-13,7) circle (2.5pt);
\draw [fill=black] (-13,5) circle (2.5pt);
\draw[color=black] (-12.84,5.2) node {2};
\draw [fill=black] (-10,7) circle (2.5pt);
\draw [fill=black] (-10,5) circle (2.5pt);
\draw[color=black] (-9.84,5.2) node {1};
\draw [fill=black] (-12,7) circle (2.5pt);
\draw [fill=black] (-15,5) circle (2.5pt);
\draw[color=black] (-14.84,5.2) node {1};
\draw [fill=black] (-11,7) circle (2.5pt);
\draw [fill=black] (-14,5) circle (2.5pt);
\draw[color=black] (-13.84,5.2) node {1};
\draw [fill=black] (-9,7) circle (2.5pt);
\draw [fill=black] (-8,5) circle (2.5pt);
\draw[color=black] (-7.84,5.2) node {2};
\draw [fill=black] (-8,7) circle (2.5pt);
\draw [fill=black] (-7,5) circle (2.5pt);
\draw[color=black] (-6.84,5.2) node {1};
\draw [fill=black] (-7,7) circle (2.5pt);
\draw [fill=black] (-6,5) circle (2.5pt);
\draw[color=black] (-5.84,5.2) node {0};
\draw [fill=black] (-6,7) circle (2.5pt);
\draw [fill=black] (-9,5) circle (2.5pt);
\draw[color=black] (-8.84,5.2) node {1};
\draw [fill=black] (-6,3) circle (2.5pt);
\draw[color=black] (-5.84,3.2) node {1};
\draw [fill=black] (-9,3) circle (2.5pt);
\draw[color=black] (-8.84,3.2) node {1};
\draw [fill=black] (-15,3) circle (2.5pt);
\draw[color=black] (-14.84,3.2) node {0};
\draw [fill=black] (-14,3) circle (2.5pt);
\draw[color=black] (-13.84,3.2) node {1};
\draw [fill=black] (-7,3) circle (2.5pt);
\draw[color=black] (-6.84,3.2) node {0};
\draw [fill=black] (-12,3) circle (2.5pt);
\draw[color=black] (-11.84,3.2) node {0};
\draw [fill=black] (-8,3) circle (2.5pt);
\draw[color=black] (-7.84,3.2) node {0};
\draw [fill=black] (-11,3) circle (2.5pt);
\draw[color=black] (-10.84,3.2) node {1};
\draw [fill=black] (-10,3) circle (2.5pt);
\draw[color=black] (-9.84,3.2) node {2};
\draw [fill=black] (-13,3) circle (2.5pt);
\draw[color=black] (-12.84,3.2) node {1};
\end{scriptsize}
\end{tikzpicture}
\]
The reader should have remarked that to obtain the $i^{th}$ component of the vector forming $xy,$ we sum up the labels following the $i^{th}$ dot of the bottom. 
\end{ex}

\begin{prop}\label{Main-Prop} Two elements of $\mathbb{Z}_k \wr \mathcal{S}_n$ are in the same $\mathbb{Z}_k \wr \mathcal{S}_{n-1}$ conjugacy class if and only if they have the same marked-type.
\end{prop}
\begin{proof} Let us show that $\mty(x)=\mty(z^{-1}xz)$ for any $x=(g;p)\in \mathbb{Z}_k \wr \mathcal{S}_n$ and any $z=(f;t)\in \mathbb{Z}_k \wr \mathcal{S}_{n-1}.$ First remark that $\ty(x)=\ty(z^{-1}xz)$ as $z$ can be considered an element of $\mathbb{Z}_k \wr \mathcal{S}_{n}$ by extending naturally $t$ to $\mathcal{S}_n$ then setting $f_n$ to be $0.$ To prove that $\mty(x)=\mty(z^{-1}xz),$ let us focus on the cycle $c$ of $p$ containing $n.$ Suppose $c=(n,c_2,\cdots , c_r)$ then $(t(n),t(c_2),\cdots , t(c_r))=(n,t(c_2),\cdots , t(c_r))$ is a cycle of $t^{-1}pt.$ The cycle sum of $(n,t(c_2),\cdots , t(c_r))$ is 
\[
f_n+g_{t^{-1}(n)}-f_{t(p^{-1}(t^{-1}(n)))}+\sum_{i=2}^r \left( f_{t(c_i)}+g_{t^{-1}(t(c_i))}-f_{t(p^{-1}(t^{-1}(t(c_i))))}\right)\]
\[=f_n+g_{n}-f_{t(c_r)}+f_{t(c_2)}+g_{c_2}-f_{n}+\sum_{i=3}^r \left( f_{t(c_i)}+g_{c_i}-f_{t(c_{i-1})}\right)
\]
\[
=g_{n}-f_{t(c_r)}+f_{t(c_2)}+g_{c_2}-f_{t(c_2)}+f_{t(c_r)}+\sum_{i=3}^rg_{c_i}=g_n+\sum_{i=2}^rg_{c_i}
\]
which is the cycle sum of $c.$

Conversely, Suppose that $x=(g;p),y=(h;q)\in \mathbb{Z}_k \wr \mathcal{S}_n$ with $\mty(x)=\mty(y).$ Then $p$ and $q$ should have the same cycle-type as permutations of $n$ with $n$ being in a cycle of the same length in both $p$ and $q.$ Take $t$ to be any permutation of $n$ such that if $c=(c_1,c_2,\cdots ,c_r)$ is a cycle of $p$ then $t(c)=(t(c_1),t(c_2),\cdots ,t(c_r))$ is a cycle of $q$ with $t(n)=n.$ Take a solution of the linear system $f_i-f_{t(p^{-1}(t^{-1}(i)))}=h_i-g_{t^{-1}(i)},$ $1\leq i\leq n$ that satisfies $f_n=0.$ Then $y=z^{-1}xz$ with $z=(f;t)\in \mathbb{Z}_k \wr \mathcal{S}_{n-1}$ which ends the proof. 
\end{proof}

\begin{ex} Recall the elements $x$ and $y$ of Examples \ref{example1} and \ref{example2}, then $$xyx^{-1}=((1,2,0,0,1,2,1,2,2,2);(1,7,2,6,10,8)(4,3,9)(5))$$ as can be obtained from the following diagram
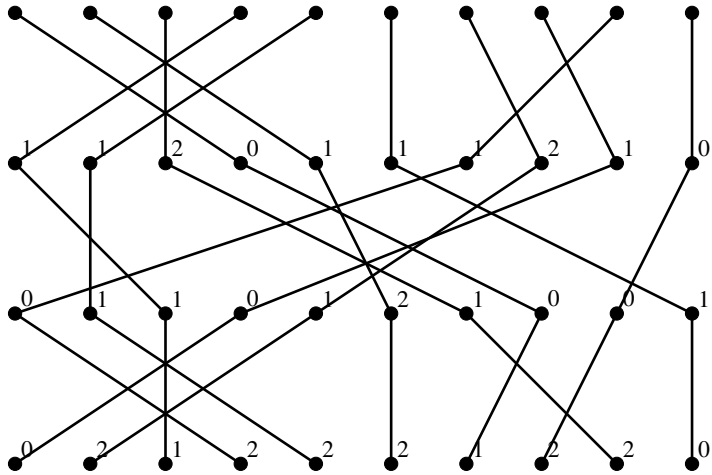
\[
\begin{tikzpicture}[line cap=round,line join=round,x=1cm,y=1cm]
\draw [line width=1pt] (-15,7)-- (-12,5);
\draw [line width=1pt] (-14,7)-- (-11,5);
\draw [line width=1pt] (-13,7)-- (-13,5);
\draw [line width=1pt] (-10,7)-- (-10,5);
\draw [line width=1pt] (-12,7)-- (-15,5);
\draw [line width=1pt] (-11,7)-- (-14,5);
\draw [line width=1pt] (-9,7)-- (-8,5);
\draw [line width=1pt] (-8,7)-- (-7,5);
\draw [line width=1pt] (-7,7)-- (-9,5);
\draw [line width=1pt] (-6,7)-- (-6,5);
\draw [line width=1pt] (-15,5)-- (-13,3);
\draw [line width=1pt] (-6,5)-- (-7,3);
\draw [line width=1pt] (-9,5)-- (-15,3);
\draw [line width=1pt] (-14,5)-- (-14,3);
\draw [line width=1pt] (-13,5)-- (-9,3);
\draw [line width=1pt] (-7,5)-- (-12,3);
\draw [line width=1pt] (-12,5)-- (-8,3);
\draw [line width=1pt] (-8,5)-- (-11,3);
\draw [line width=1pt] (-11,5)-- (-10,3);
\draw [line width=1pt] (-10,5)-- (-6,3);
\draw [line width=1pt] (-15,3)-- (-12,1);
\draw [line width=1pt] (-12,3)-- (-15,1);
\draw [line width=1pt] (-14,3)-- (-11,1);
\draw [line width=1pt] (-13,3)-- (-13,1);
\draw [line width=1pt] (-11,3)-- (-14,1);
\draw [line width=1pt] (-10,3)-- (-10,1);
\draw [line width=1pt] (-9,3)-- (-7,1);
\draw [line width=1pt] (-8,3)-- (-9,1);
\draw [line width=1pt] (-7,3)-- (-8,1);
\draw [line width=1pt] (-6,3)-- (-6,1);
\begin{scriptsize}
\draw [fill=black] (-15,7) circle (2.5pt);
\draw [fill=black] (-12,5) circle (2.5pt);
\draw[color=black] (-11.84,5.2) node {0};
\draw [fill=black] (-14,7) circle (2.5pt);
\draw [fill=black] (-11,5) circle (2.5pt);
\draw[color=black] (-10.84,5.2) node {1};
\draw [fill=black] (-13,7) circle (2.5pt);
\draw [fill=black] (-13,5) circle (2.5pt);
\draw[color=black] (-12.84,5.2) node {2};
\draw [fill=black] (-10,7) circle (2.5pt);
\draw [fill=black] (-10,5) circle (2.5pt);
\draw[color=black] (-9.84,5.2) node {1};
\draw [fill=black] (-12,7) circle (2.5pt);
\draw [fill=black] (-15,5) circle (2.5pt);
\draw[color=black] (-14.84,5.2) node {1};
\draw [fill=black] (-11,7) circle (2.5pt);
\draw [fill=black] (-14,5) circle (2.5pt);
\draw[color=black] (-13.84,5.2) node {1};
\draw [fill=black] (-9,7) circle (2.5pt);
\draw [fill=black] (-8,5) circle (2.5pt);
\draw[color=black] (-7.84,5.2) node {2};
\draw [fill=black] (-8,7) circle (2.5pt);
\draw [fill=black] (-7,5) circle (2.5pt);
\draw[color=black] (-6.84,5.2) node {1};
\draw [fill=black] (-7,7) circle (2.5pt);
\draw [fill=black] (-6,5) circle (2.5pt);
\draw[color=black] (-5.84,5.2) node {0};
\draw [fill=black] (-6,7) circle (2.5pt);
\draw [fill=black] (-9,5) circle (2.5pt);
\draw[color=black] (-8.84,5.2) node {1};
\draw [fill=black] (-6,3) circle (2.5pt);
\draw[color=black] (-5.84,3.2) node {1};
\draw [fill=black] (-9,3) circle (2.5pt);
\draw[color=black] (-8.84,3.2) node {1};
\draw [fill=black] (-15,3) circle (2.5pt);
\draw[color=black] (-14.84,3.2) node {0};
\draw [fill=black] (-14,3) circle (2.5pt);
\draw[color=black] (-13.84,3.2) node {1};
\draw [fill=black] (-7,3) circle (2.5pt);
\draw[color=black] (-6.84,3.2) node {0};
\draw [fill=black] (-12,3) circle (2.5pt);
\draw[color=black] (-11.84,3.2) node {0};
\draw [fill=black] (-8,3) circle (2.5pt);
\draw[color=black] (-7.84,3.2) node {0};
\draw [fill=black] (-11,3) circle (2.5pt);
\draw[color=black] (-10.84,3.2) node {1};
\draw [fill=black] (-10,3) circle (2.5pt);
\draw[color=black] (-9.84,3.2) node {2};
\draw [fill=black] (-13,3) circle (2.5pt);
\draw[color=black] (-12.84,3.2) node {1};
\draw [fill=black] (-12,1) circle (2.5pt);
\draw[color=black] (-11.84,1.2) node {2};
\draw [fill=black] (-15,1) circle (2.5pt);
\draw[color=black] (-14.84,1.2) node {0};
\draw [fill=black] (-11,1) circle (2.5pt);
\draw[color=black] (-10.84,1.2) node {2};
\draw [fill=black] (-13,1) circle (2.5pt);
\draw[color=black] (-12.84,1.2) node {1};
\draw [fill=black] (-14,1) circle (2.5pt);
\draw[color=black] (-13.84,1.2) node {2};
\draw [fill=black] (-10,1) circle (2.5pt);
\draw[color=black] (-9.84,1.2) node {2};
\draw [fill=black] (-6,1) circle (2.5pt);
\draw[color=black] (-5.84,1.2) node {0};
\draw [fill=black] (-9,1) circle (2.5pt);
\draw[color=black] (-8.84,1.2) node {1};
\draw [fill=black] (-8,1) circle (2.5pt);
\draw[color=black] (-7.84,1.2) node {2};
\draw [fill=black] (-7,1) circle (2.5pt);
\draw[color=black] (-6.84,1.2) node {2};
\end{scriptsize}
\end{tikzpicture}
\]
We leave it to the reader to check that $\mty(xyx^{-1})=\mty(y)=(\emptyset,(6^\ast,1),(3)).$
\end{ex}

\begin{ex}
Consider the elements $x=((1,0,2,1,1,0,2);(1,2,3)(4,5)(6,7))$ and $y=((0,2,1,0,0,0,1);(1,4,5)(2,6)(3,7))$ of $\mathbb{Z}_3 \wr \mathcal{S}_7.$ They both have the same marked-type $((3),\emptyset,(2,2^\ast)).$ Let $t=(1)(2,4)(3,5,6)(7).$ The linear system $f_i-f_{t(p^{-1}(t^{-1}(i)))}=h_i-g_{t^{-1}(i)},$ $1\leq i\leq n$ that satisfies $f_n=0,$ has the following equations:
\[
\systeme{f_1-f_5=2,f_2-f_6=1,f_3-f_7=1,f_4-f_1=0,f_5-f_4=1,f_6-f_2=2,f_7-f_3=2}
\]
Setting $f_7=0,$ the system has many solutions, for example $f_1=1,$ $f_2=1,$ $f_3=1,$ $f_4=1,$ $f_5=2,$ $f_6=0$ and $f_7=0.$ Thus $y=z^{-1}xz$ for $z=((1,1,1,1,2,0,0);(1)(2,4)(3,5,6)(7))\in \mathbb{Z}_3 \wr \mathcal{S}_6$ which can be verified by drawing the following diagram of $z^{-1}xz$
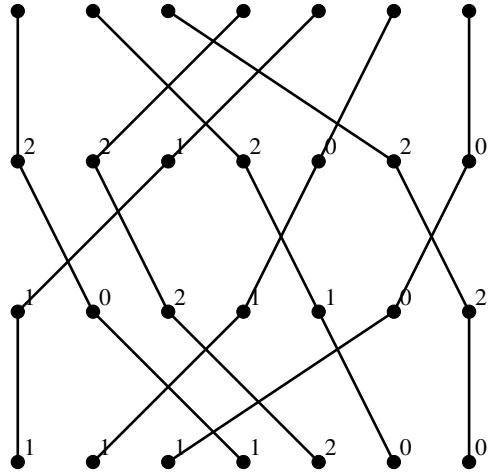
\[
\begin{tikzpicture}[line cap=round,line join=round,x=1cm,y=1cm]
\draw [line width=1pt] (-15,7)-- (-15,5);
\draw [line width=1pt] (-14,7)-- (-12,5);
\draw [line width=1pt] (-13,7)-- (-10,5);
\draw [line width=1pt] (-10,7)-- (-11,5);
\draw [line width=1pt] (-12,7)-- (-14,5);
\draw [line width=1pt] (-11,7)-- (-13,5);
\draw [line width=1pt] (-9,7)-- (-9,5);
\draw [line width=1pt] (-15,5)-- (-14,3);
\draw [line width=1pt] (-9,5)-- (-10,3);
\draw [line width=1pt] (-14,5)-- (-13,3);
\draw [line width=1pt] (-13,5)-- (-15,3);
\draw [line width=1pt] (-12,5)-- (-11,3);
\draw [line width=1pt] (-11,5)-- (-12,3);
\draw [line width=1pt] (-10,5)-- (-9,3);
\draw [line width=1pt] (-15,3)-- (-15,1);
\draw [line width=1pt] (-12,3)-- (-14,1);
\draw [line width=1pt] (-14,3)-- (-12,1);
\draw [line width=1pt] (-13,3)-- (-11,1);
\draw [line width=1pt] (-11,3)-- (-10,1);
\draw [line width=1pt] (-10,3)-- (-13,1);
\draw [line width=1pt] (-9,3)-- (-9,1);
\begin{scriptsize}
\draw [fill=black] (-15,7) circle (2.5pt);
\draw [fill=black] (-12,5) circle (2.5pt);
\draw[color=black] (-11.84,5.2) node {2};
\draw [fill=black] (-14,7) circle (2.5pt);
\draw [fill=black] (-11,5) circle (2.5pt);
\draw[color=black] (-10.84,5.2) node {0};
\draw [fill=black] (-13,7) circle (2.5pt);
\draw [fill=black] (-13,5) circle (2.5pt);
\draw[color=black] (-12.84,5.2) node {1};
\draw [fill=black] (-10,7) circle (2.5pt);
\draw [fill=black] (-10,5) circle (2.5pt);
\draw[color=black] (-9.84,5.2) node {2};
\draw [fill=black] (-12,7) circle (2.5pt);
\draw [fill=black] (-15,5) circle (2.5pt);
\draw[color=black] (-14.84,5.2) node {2};
\draw [fill=black] (-11,7) circle (2.5pt);
\draw [fill=black] (-14,5) circle (2.5pt);
\draw[color=black] (-13.84,5.2) node {2};
\draw [fill=black] (-9,7) circle (2.5pt);
\draw [fill=black] (-9,5) circle (2.5pt);
\draw[color=black] (-8.84,5.2) node {0};
\draw [fill=black] (-9,3) circle (2.5pt);
\draw[color=black] (-8.84,3.2) node {2};
\draw [fill=black] (-15,3) circle (2.5pt);
\draw[color=black] (-14.84,3.2) node {1};
\draw [fill=black] (-14,3) circle (2.5pt);
\draw[color=black] (-13.84,3.2) node {0};
\draw [fill=black] (-12,3) circle (2.5pt);
\draw[color=black] (-11.84,3.2) node {1};
\draw [fill=black] (-11,3) circle (2.5pt);
\draw[color=black] (-10.84,3.2) node {1};
\draw [fill=black] (-10,3) circle (2.5pt);
\draw[color=black] (-9.84,3.2) node {0};
\draw [fill=black] (-13,3) circle (2.5pt);
\draw[color=black] (-12.84,3.2) node {2};
\draw [fill=black] (-12,1) circle (2.5pt);
\draw[color=black] (-11.84,1.2) node {1};
\draw [fill=black] (-15,1) circle (2.5pt);
\draw[color=black] (-14.84,1.2) node {1};
\draw [fill=black] (-11,1) circle (2.5pt);
\draw[color=black] (-10.84,1.2) node {2};
\draw [fill=black] (-13,1) circle (2.5pt);
\draw[color=black] (-12.84,1.2) node {1};
\draw [fill=black] (-14,1) circle (2.5pt);
\draw[color=black] (-13.84,1.2) node {1};
\draw [fill=black] (-10,1) circle (2.5pt);
\draw[color=black] (-9.84,1.2) node {0};
\draw [fill=black] (-9,1) circle (2.5pt);
\draw[color=black] (-8.84,1.2) node {0};
\end{scriptsize}
\end{tikzpicture}
\]
\end{ex}

\begin{remark} The reader should check that any other solution for the linear system in the above example will give us a valid $z$ that satisfies $y=z^{-1}xz.$ For example $$z=((0,0,1,0,1,2,0);(1)(2,4)(3,5,6)(7))\in \mathbb{Z}_3 \wr \mathcal{S}_6.$$
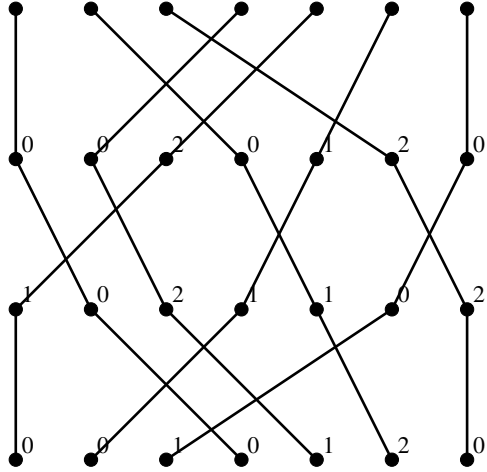
\[
\begin{tikzpicture}[line cap=round,line join=round,x=1cm,y=1cm]
\draw [line width=1pt] (-15,7)-- (-15,5);
\draw [line width=1pt] (-14,7)-- (-12,5);
\draw [line width=1pt] (-13,7)-- (-10,5);
\draw [line width=1pt] (-10,7)-- (-11,5);
\draw [line width=1pt] (-12,7)-- (-14,5);
\draw [line width=1pt] (-11,7)-- (-13,5);
\draw [line width=1pt] (-9,7)-- (-9,5);
\draw [line width=1pt] (-15,5)-- (-14,3);
\draw [line width=1pt] (-9,5)-- (-10,3);
\draw [line width=1pt] (-14,5)-- (-13,3);
\draw [line width=1pt] (-13,5)-- (-15,3);
\draw [line width=1pt] (-12,5)-- (-11,3);
\draw [line width=1pt] (-11,5)-- (-12,3);
\draw [line width=1pt] (-10,5)-- (-9,3);
\draw [line width=1pt] (-15,3)-- (-15,1);
\draw [line width=1pt] (-12,3)-- (-14,1);
\draw [line width=1pt] (-14,3)-- (-12,1);
\draw [line width=1pt] (-13,3)-- (-11,1);
\draw [line width=1pt] (-11,3)-- (-10,1);
\draw [line width=1pt] (-10,3)-- (-13,1);
\draw [line width=1pt] (-9,3)-- (-9,1);
\begin{scriptsize}
\draw [fill=black] (-15,7) circle (2.5pt);
\draw [fill=black] (-12,5) circle (2.5pt);
\draw[color=black] (-11.84,5.2) node {0};
\draw [fill=black] (-14,7) circle (2.5pt);
\draw [fill=black] (-11,5) circle (2.5pt);
\draw[color=black] (-10.84,5.2) node {1};
\draw [fill=black] (-13,7) circle (2.5pt);
\draw [fill=black] (-13,5) circle (2.5pt);
\draw[color=black] (-12.84,5.2) node {2};
\draw [fill=black] (-10,7) circle (2.5pt);
\draw [fill=black] (-10,5) circle (2.5pt);
\draw[color=black] (-9.84,5.2) node {2};
\draw [fill=black] (-12,7) circle (2.5pt);
\draw [fill=black] (-15,5) circle (2.5pt);
\draw[color=black] (-14.84,5.2) node {0};
\draw [fill=black] (-11,7) circle (2.5pt);
\draw [fill=black] (-14,5) circle (2.5pt);
\draw[color=black] (-13.84,5.2) node {0};
\draw [fill=black] (-9,7) circle (2.5pt);
\draw [fill=black] (-9,5) circle (2.5pt);
\draw[color=black] (-8.84,5.2) node {0};
\draw [fill=black] (-9,3) circle (2.5pt);
\draw[color=black] (-8.84,3.2) node {2};
\draw [fill=black] (-15,3) circle (2.5pt);
\draw[color=black] (-14.84,3.2) node {1};
\draw [fill=black] (-14,3) circle (2.5pt);
\draw[color=black] (-13.84,3.2) node {0};
\draw [fill=black] (-12,3) circle (2.5pt);
\draw[color=black] (-11.84,3.2) node {1};
\draw [fill=black] (-11,3) circle (2.5pt);
\draw[color=black] (-10.84,3.2) node {1};
\draw [fill=black] (-10,3) circle (2.5pt);
\draw[color=black] (-9.84,3.2) node {0};
\draw [fill=black] (-13,3) circle (2.5pt);
\draw[color=black] (-12.84,3.2) node {2};
\draw [fill=black] (-12,1) circle (2.5pt);
\draw[color=black] (-11.84,1.2) node {0};
\draw [fill=black] (-15,1) circle (2.5pt);
\draw[color=black] (-14.84,1.2) node {0};
\draw [fill=black] (-11,1) circle (2.5pt);
\draw[color=black] (-10.84,1.2) node {1};
\draw [fill=black] (-13,1) circle (2.5pt);
\draw[color=black] (-12.84,1.2) node {1};
\draw [fill=black] (-14,1) circle (2.5pt);
\draw[color=black] (-13.84,1.2) node {0};
\draw [fill=black] (-10,1) circle (2.5pt);
\draw[color=black] (-9.84,1.2) node {2};
\draw [fill=black] (-9,1) circle (2.5pt);
\draw[color=black] (-8.84,1.2) node {0};
\end{scriptsize}
\end{tikzpicture}
\]
\end{remark}

\begin{cor}\label{cor:sym_Gelf_pair} The pair $(\mathbb{Z}_k \wr \mathcal{S}_n \times \mathbb{Z}_k \wr \mathcal{S}_{n-1}, \diag (\mathbb{Z}_k \wr \mathcal{S}_{n-1}) )$ is a symmetric Gelfand pair.
\end{cor}
\begin{proof}
It would be clear from the cycle decomposition that $\mty(x)=\mty(x^{-1})$ for any $x\in \mathbb{Z}_k \wr \mathcal{S}_n$ which implies that $x$ and $x^{-1}$ are $\mathbb{Z}_k \wr \mathcal{S}_{n-1}$-conjugate by Proposition \ref{Main-Prop}. The result then follows from Proposition \ref{prop_couple_sym_Gel}.
\end{proof}

As a consequence of Proposition \ref{Main-Prop}, the $\mathbb{Z}_k \wr \mathcal{S}_{n-1}$-conjugacy classes of  $\mathbb{Z}_k \wr \mathcal{S}_{n}$ can be index by marked $k$-partite partitions of $n.$ If $\blambda^{(j,\lambda_i)}$ is a marked $k$-partite partitions, then its associated $\mathbb{Z}_k \wr \mathcal{S}_{n-1}$-conjugacy class is 
\[
C_{\blambda^{(j,\lambda_i)}}:=\lbrace x\in \mathbb{Z}_k \wr \mathcal{S}_{n} \mid \mty(x)=\blambda^{(j,\lambda_i)}\rbrace.
\]

\section{generalized characters of $\mathbb{Z}_k \wr \mathcal{S}_{n}$} 

In the previous section, we showed in Corollary \ref{cor:sym_Gelf_pair} that the pair $(\mathbb{Z}_k \wr \mathcal{S}_n \times \mathbb{Z}_k \wr \mathcal{S}_{n-1}, \diag (\mathbb{Z}_k \wr \mathcal{S}_{n-1}) )$ is a symmetric Gelfand pair. This implies that $1_{\diag (\mathbb{Z}_k \wr \mathcal{S}_{n-1})}^{\mathbb{Z}_k \wr \mathcal{S}_n \times \mathbb{Z}_k \wr \mathcal{S}_{n-1}}$ is multiplicity free. Its character is given explicitly in the following proposition.

If $f$ and $g$ are two functions defined on the group $\mathbb{Z}_k \wr \mathcal{S}_{n},$ then their scalar product is defined by:
\[
\langle f,g\rangle_{\mathbb{Z}_k \wr \mathcal{S}_{n}}=\frac{1}{|\mathbb{Z}_k \wr \mathcal{S}_{n}|}\sum_{x\in \mathbb{Z}_k \wr \mathcal{S}_{n}} f(x)\overline{g(x)}.
\]
\begin{prop}
The induced representation $1_{\diag (\mathbb{Z}_k \wr \mathcal{S}_{n-1})}^{\mathbb{Z}_k \wr \mathcal{S}_n \times \mathbb{Z}_k \wr \mathcal{S}_{n-1}}$ is multiplicity free and its character is:
\begin{equation}
\character \left(1_{\diag (\mathbb{Z}_k \wr \mathcal{S}_{n-1})}^{\mathbb{Z}_k \wr \mathcal{S}_n \times \mathbb{Z}_k \wr \mathcal{S}_{n-1}}\right)=\sum_{\brho \text{$k$-partite partition of $n,$} \atop{ \bsigma\nearrow \brho }}\chi^{\brho}\times \chi^{\bsigma}.
\end{equation}
\end{prop}
\begin{proof} Take any irreducible character of the product group $\mathbb{Z}_k \wr \mathcal{S}_n \times \mathbb{Z}_k \wr \mathcal{S}_{n-1}$ which has the form $\chi^{\blambda}\times \chi^{\bdelta}$ where $\blambda$ is a $k$-partite partition of $n$ and $\bdelta$ is a $k$-partite partition of $n-1.$ By the Frobenius reciprocity theorem (see for example \cite[Theorem 1.12.6]{sagan2001symmetric}) we have:
\begin{equation}\label{hyp_frob}
\langle\ \chi^{\blambda}\times \chi^{\bdelta},\character \left(1_{\diag (\mathbb{Z}_k \wr \mathcal{S}_{n-1})}^{\mathbb{Z}_k \wr \mathcal{S}_n \times \mathbb{Z}_k \wr \mathcal{S}_{n-1}}\right)\rangle\ _{\mathbb{Z}_k \wr \mathcal{S}_n \times \mathbb{Z}_k \wr \mathcal{S}_{n-1}}
=\langle\ \Res\left(\chi^{\blambda}\times \chi^{\bdelta}\right),1\rangle\ _{\diag (\mathbb{Z}_k \wr \mathcal{S}_{n-1})}
\end{equation}
\begin{equation*}
=\frac{1}{|\mathbb{Z}_k \wr \mathcal{S}_{n-1}|}\sum_{x\in \mathbb{Z}_k \wr \mathcal{S}_{n-1}}\Res_{\mathbb{Z}_k \wr \mathcal{S}_{n-1}}^{\mathbb{Z}_k \wr \mathcal{S}_{n}}\chi^{\blambda}(x) \chi^{\bdelta}(x),
\end{equation*}
where $\Res_{\mathbb{Z}_k \wr \mathcal{S}_{n-1}}^{\mathbb{Z}_k \wr \mathcal{S}_{n}}\chi^{\blambda}$ denotes the restriction of the character $\chi^{\blambda}$ of $\mathbb{Z}_k \wr \mathcal{S}_{n}$ to $\mathbb{Z}_k \wr \mathcal{S}_{n-1}.$
But 
\begin{equation*}
\Res_{\mathbb{Z}_k \wr \mathcal{S}_{n-1}}^{\mathbb{Z}_k \wr \mathcal{S}_{n}}\chi^{\blambda}=\sum_{\bsigma\nearrow \blambda}\chi^{\bsigma},
\end{equation*}
as it is shown in \cite[Corollary 5.3]{stein2017littlewood}. Thus, using the orthogonality property of the characters of $\mathbb{Z}_k \wr \mathcal{S}_{n-1},$ Equation (\ref{hyp_frob}) becomes:
\begin{equation}
\langle\ \chi^{\blambda}\times \chi^{\bdelta},\character \left(1_{\diag (\mathbb{Z}_k \wr \mathcal{S}_{n-1})}^{\mathbb{Z}_k \wr \mathcal{S}_n \times \mathbb{Z}_k \wr \mathcal{S}_{n-1}}\right)\rangle\ _{\mathbb{Z}_k \wr \mathcal{S}_n \times \mathbb{Z}_k \wr \mathcal{S}_{n-1}}
=\begin{cases}
1 & \text{ if $\bdelta \nearrow \blambda$}\\
0 & \text{ otherwise.}
\end{cases}
\end{equation}
On the other hand, we have
\begin{equation}
\langle\ \chi^{\blambda}\times \chi^{\bdelta},\sum_{\brho \text{$k$-partite partition of $n,$} \atop{ \bsigma\nearrow \brho }}\chi^{\brho}\times \chi^{\bsigma}\rangle\ _{\mathbb{Z}_k \wr \mathcal{S}_n \times \mathbb{Z}_k \wr \mathcal{S}_{n-1}}
\end{equation}
\begin{equation*}
=\sum_{\brho \text{$k$-partite partition of $n,$} \atop{ \bsigma\nearrow \brho }}\langle\ \chi^{\blambda},\chi^{\brho}\rangle\ _{\mathbb{Z}_k \wr \mathcal{S}_{n}}\langle\  \chi^{\bdelta},\chi^{\bsigma}\rangle\ _{\mathbb{Z}_k \wr \mathcal{S}_{n-1}}
=\begin{cases}
1 & \text{ if $\bdelta \nearrow \blambda$}\\
0 & \text{ otherwise.}
\end{cases}
\end{equation*}
This ends the proof of this proposition.

\end{proof}

\begin{cor} The zonal spherical functions of the pair $(\mathbb{Z}_k \wr \mathcal{S}_n \times \mathbb{Z}_k \wr \mathcal{S}_{n-1}, \diag (\mathbb{Z}_k \wr \mathcal{S}_{n-1}) )$ are
\begin{equation*}
\omega^{\bsigma\nearrow \brho}(x,y)=\frac{1}{k^{n-1}(n-1)!}\sum_{h\in \mathbb{Z}_k \wr \mathcal{S}_{n-1}}\chi^{\brho}(xh)\chi^{\bsigma}(yh),
\end{equation*}
where $\bsigma$ is a $k$-partite partition of $n-1$ and $\brho$ is a $k$-partite partition of $n$ with $\bsigma\nearrow \brho.$
\end{cor}

The $\mathbb{Z}_k \wr \mathcal{S}_{n-1}$-generalized characters of $\mathbb{Z}_k \wr \mathcal{S}_{n}$ associated to $\bsigma\nearrow \brho,$ where $\brho$ is a $k$-partite partition of $n,$ is defined by
\begin{equation}\label{gen_char_hyp}
\chi^{\bsigma\nearrow \brho}(x)=\chi^{\bsigma}(1)\omega^{\bsigma\nearrow \brho}(x,1)= \frac{\chi^{\bsigma}(1)}{k^{n-1}(n-1)!}\sum_{h\in \mathbb{Z}_k \wr \mathcal{S}_{n-1}}\chi^{\brho}(xh)\chi^{\bsigma}(h) 
\end{equation}
for any $x\in \mathbb{Z}_k \wr \mathcal{S}_{n}.$ 

The zonal spherical functions of the Gelfand pair $(\mathbb{Z}_k \wr \mathcal{S}_n \times \mathbb{Z}_k \wr \mathcal{S}_{n-1}, \diag (\mathbb{Z}_k \wr \mathcal{S}_{n-1}) )$ form an orthogonal basis for $L(\diag (\mathbb{Z}_k \wr \mathcal{S}_{n-1})\backslash \mathbb{Z}_k \wr \mathcal{S}_{n}\times\mathbb{Z}_k \wr \mathcal{S}_{n-1}/\diag (\mathbb{Z}_k \wr \mathcal{S}_{n-1})),$ the algebra of all complex-valued functions on $\mathbb{Z}_k \wr \mathcal{S}_{n}\times\mathbb{Z}_k \wr \mathcal{S}_{n-1}$ that are constant on the double-classes $\diag (\mathbb{Z}_k \wr \mathcal{S}_{n-1}) x\diag (\mathbb{Z}_k \wr \mathcal{S}_{n-1})$ of $\mathbb{Z}_k \wr \mathcal{S}_{n}\times\mathbb{Z}_k \wr \mathcal{S}_{n-1}.$ The list of basic important properties of zonal spherical functions can be found in \cite[page 389]{McDo}. As the $\mathbb{Z}_k \wr \mathcal{S}_{n-1}$-generalized characters of $\mathbb{Z}_k \wr \mathcal{S}_{n}$ are closely related, by their definition, to the zonal spherical functions of $(\mathbb{Z}_k \wr \mathcal{S}_{n}\times\mathbb{Z}_k \wr \mathcal{S}_{n-1}, \diag (\mathbb{Z}_k \wr \mathcal{S}_{n-1}) ),$ we list below their properties that can be obtained directly from those of the zonal spherical functions:
\begin{enumerate}
\item[•] They form an orthogonal basis for the algebra $C(\mathbb{Z}_k \wr \mathcal{S}_{n}\times\mathbb{Z}_k \wr \mathcal{S}_{n-1},\diag (\mathbb{Z}_k \wr \mathcal{S}_{n-1}))$ of all complex-valued functions on $G$ that are constant on the $\diag (\mathbb{Z}_k \wr \mathcal{S}_{n-1})-$ conjugacy classes of $\mathbb{Z}_k \wr \mathcal{S}_{n}\times\mathbb{Z}_k \wr \mathcal{S}_{n-1}.$
\item[•] $\chi^{\bsigma\nearrow \brho}(1)=\chi^{\bsigma}(1).$
\item[•] $\chi^{\bsigma\nearrow \brho}(yxy^{-1})=\chi^{\bsigma\nearrow \brho}(x)$ for all $x\in \mathbb{Z}_k \wr \mathcal{S}_{n},$ $y\in \mathbb{Z}_k \wr \mathcal{S}_{n-1}.$
\item[•] $\langle \chi^{\bmu\nearrow \blambda},\chi^{\bnu\nearrow \brho}\rangle_{\mathbb{Z}_k \wr \mathcal{S}_{n}}=\frac{\chi^{\bmu}(1)}{\chi^{\blambda}(1)}\delta^{\blambda\brho}\delta^{\bmu\bnu}.$
\end{enumerate}

When $k=1,$ $\mathbb{Z}_k \wr \mathcal{S}_{n}$ is isomorphic to $\mathcal{S}_{n}$ and tables of the generalized characters of $\mathcal{S}_{3}$ and $\mathcal{S}_{4}$ can be found in \cite[page 118]{strahov2007generalized}. As it seems to us that the tables entries in \cite[page 118]{strahov2007generalized} are not totally correct, we reproduce them below 
\begin{table}[h]
  \[
  \begin{array}{|*{5}{c|}}
    \hline
&   \boldmath C_{(3^*)} & \boldmath C_{(2,1^*)} & \boldmath C_{(2^*,1)} & \boldmath C_{(1,1,1^*)} \\
    \hline
\text{Order} &   2 & 1 & 2 & 1  \\
	\hline
\text{Elements} &    (123),(132) &  (12)(3)  &  (13)(2),(23)(1) &  (1)(2)(3)  \\
	\hline
\chi^{(3^*)}    &   1 & 1 & 1 & 1  \\
    \hline
\chi^{(2,1^*)}    &   -\frac{1}{2} & 1 & -\frac{1}{2} & 1  \\
    \hline
\chi^{(2^*,1)}    &   -\frac{1}{2} & -1 & \frac{1}{2} & 1  \\
    \hline
\chi^{(1,1,1^*)}    &   1 & -1 & -1 & 1  \\
	\hline
  \end{array}
  \]
  \caption{$\mathcal{S}_2$-generalized characters of $\mathcal{S}_3$}
  \label{S2 gen char tab}
\end{table}

\begin{table}[h]
  \[
  \begin{array}{|*{8}{c|}}
    \hline
&   \boldmath C_{(4^*)} & \boldmath C_{(3,1^*)} & \boldmath C_{(3^*,1)} & \boldmath C_{(2,2^*)} & \boldmath C_{(2,1,1^*)} & \boldmath C_{(2^*,1,1)} & \boldmath C_{(1,1,1,1^*)}\\
    \hline
\text{Order} &   6 & 2 & 6 & 3 & 3 & 3 & 1 \\
	\hline
\chi^{(4^*)}    &   1 & 1 & 1 & 1 & 1 & 1 & 1 \\
    \hline
\chi^{(3,1^*)}    &   -\frac{1}{3} & 1 & -\frac{1}{3} & -\frac{1}{3} & 1 & -\frac{1}{3} & 1 \\
    \hline
\chi^{(3^*,1)}    &   -\frac{2}{3} & -1 & \frac{1}{3} & -\frac{2}{3} & 0 & \frac{4}{3} & 2 \\
    \hline
\chi^{(2,2^*)}    &   0 & -1 & -1 & 2 & 0 & 0 & 2 \\
    \hline
\chi^{(2,1,1^*)}    &   \frac{2}{3} & -1 & \frac{1}{3} & -\frac{2}{3} & 0 & -\frac{4}{3} & 2 \\
    \hline
\chi^{(2^*,1,1)}    &   \frac{1}{3} & 1 & -\frac{1}{3} & -\frac{1}{3} & -1 & \frac{1}{3} & 1 \\
    \hline
\chi^{(1,1,1,1^*)}    &   -1 & 1 & 1 & 1 & -1 & -1 & 1 \\
    \hline
  \end{array}
  \]
  \caption{$\mathcal{S}_3$-generalized characters of $\mathcal{S}_4$}
  \label{S3 gen char tab}
\end{table}

When $k=2,$ $\mathbb{Z}_k \wr \mathcal{S}_{n}$ is isomorphic to the Hyperoctahedral group $\mathcal{H}_n,$ and the generalized characters in this case were first defined in \cite{tout2021symmetric}. As we will focus on the Hyperoctahedral groups in our examples, we refer to \cite[Pages 51, 52 and 53]{islami2021symmetric} for the character tables of $\mathcal{H}_1,$ $\mathcal{H}_2$ and $\mathcal{H}_3.$ For example, by definition $\chi^{((1^*),(1))}((13)(24))$ equals
\[
\frac{\chi^{(\emptyset,(1))}(1)}{2^{2-1}(2-1)!}\big(\chi^{((1),(1))}((13)(24))\chi^{(\emptyset,(1))}((1)(2)) + \chi^{((1),(1))}((13)(24)(12))\chi^{(\emptyset,(1))}((12)) \big)
\]
\[
=\frac{1}{2}\big(0.1 + 0.-1 \big)=0.
\]
Similarly,
\[
\chi^{(\emptyset,(2^\star))}((12)(3)(4))=\frac{1}{2}\big(-1.1 + 1.-1 \big)=-1.
\]

We use the character tables of $\mathcal{H}_1$ and $\mathcal{H}_2$ along with the above properties to produce, in Table \ref{H1 gen char tab}, the $\mathcal{H}_1$-generalized characters of $\mathcal{H}_2.$

\begin{table}[h]
  \[
  \begin{array}{|*{7}{c|}}
    \hline
&   \boldmath C_{((1,1^*),\emptyset)} & \boldmath C_{((2^*),\emptyset)} & \boldmath C_{((1^*),(1))} & \boldmath C_{((1),(1^*))} & \boldmath C_{(\emptyset,(1,1^*))} & \boldmath C_{(\emptyset,(2^*))} \\
    \hline
\text{Order} &   1 & 2 & 1 & 1 & 1 & 2 \\
	\hline
\text{Elements} &    (1)(2)(3)(4) &  (13)(24),(14)(23)  &  (12)(3)(4) &  (1)(2)(34) &  (12)(34) &  (1324),(1423) \\
	\hline
\chi^{((2^*),\emptyset})    &   1 & 1 & 1 & 1 & 1 & 1 \\
    \hline
\chi^{((1,1^*),\emptyset)}    &   1 & -1 & 1 & 1 & 1 & -1 \\
    \hline
\chi^{((1^*),(1))}    &   1 & 0 & -1 & 1 & -1 & 0 \\
    \hline
\chi^{((1),(1^*))}    &   1 & 0 & 1 & -1 & -1 & 0 \\
    \hline
\chi^{(\emptyset,(2^*))}    &   1 & 1 & -1 & -1 & 1 & -1 \\
    \hline
\chi^{(\emptyset,(1,1^*))}    &   1 & -1 & -1 & -1 & 1 & 1 \\
    \hline
  \end{array}
  \]
  \caption{$\mathcal{H}_1$-generalized characters of $\mathcal{H}_2$}
  \label{H1 gen char tab}
\end{table}

\section{Murnaghan-Nakayama rule for the generalized characters of the Hyperoctahedral group}

Let $\blambda$ be a $2$-partite partition, also called bipartition, of $n.$ We denote by $\chi^{\blambda}$ its corresponding irreducible character of $\mathbb{Z}_2\wr \mathcal{S}_n.$ A bipartite rim hook tableau of shape $\blambda$ is a sequence
\[
\emptyset=\blambda^{(0)}\subseteq \blambda^{(1)}\subseteq\cdots\subseteq\blambda^{(t)}=\blambda
\] 
of bipartite partitions such that each consecutive difference ${\bf rh_i}:=\blambda^{(i)}/\blambda^{(i-1)}$ $(1 \leq i \leq t),$ has exactly one empty part and one nonempty part $rh_i$ which is a rim hook. It would be useful for a bipartite partition $\brho$ of $n$ to define $\Gather(\brho):=(\rho^1,\cdots ,\rho^k)$ to be the partition of $n$ formed using the parts of the component partitions of $\brho.$ For example, if $\brho=((2,1),(2))$ then $\Gather(\brho)=(2,2,1).$

\begin{prop}\label{Murnaghan_Nakayama_rule}(Murnaghan-Nakayama rule for the characters of $\mathbb{Z}_2\wr \mathcal{S}_n$) If $\blambda$ and $\brho$ are two bipartite partitions of $n$ with $\Gather(\brho)=(\rho^1,\cdots,\rho^m)$ then
\[
\chi^{\blambda}_{\brho}=\sum_{S} \prod_{i=1}^{m}(-1)^{\height(rh_i)}.(-1)^{f(i).z(\rho^i)}
\]
where the sum is taken over all sequences $S$ of bipartite rim hook tableau of shape $\blambda$ such that $l(rh_i)=\rho^i$ for all $i;$ $f(i) \in \lbrace 0, 1\rbrace$ is the index of the nonempty part $rh_i$ of ${\bf rh_i}$ and $z(\rho^i)$ is the index of the component partition of $\brho$ from which $\rho^i$ is taken.
\end{prop}

\begin{ex} to compute $\chi^{((1^2);(1))}_{((1^2);(1))},$ the following three sequences of $2$-partite rim hook tableau of shape $ ((1^2),(1))$ should be considered 
\[S_1=\emptyset\subseteq ((1);\emptyset)\subseteq ((1^2);\emptyset)\subseteq  ((1^2);(1)),\]
\[S_2=\emptyset\subseteq ((1);\emptyset)\subseteq ((1);(1))\subseteq  ((1^2);(1)),\]
\[S_3=\emptyset\subseteq (\emptyset;(1))\subseteq (\emptyset;(1^2))\subseteq  ((1^2);(1))\]
These sequences will contribute to the Murnaghan-Nakayama rule respectively by the following amounts:
\[(-1)^0.(-1)^{0.0}\times (-1)^0.(-1)^{0.0}\times (-1)^0.(-1)^{1.1}=-1,\]
\[(-1)^0.(-1)^{0.0}\times (-1)^0.(-1)^{1.0}\times (-1)^0.(-1)^{0.1}=1,\]
\[(-1)^0.(-1)^{1.0}\times (-1)^0.(-1)^{0.0}\times (-1)^0.(-1)^{0.1}=1.\]
Thus,
\[ \chi^{((1^2);(1))}_{((1^2);(1))}=-1+1+1=1.\]
Using the same sequences, one can verify that \[ \chi^{((1^2);(1))}_{(\emptyset;(1^3))}=-1-1-1=-3.\]
\end{ex}

\begin{ex} By the Murnaghan-Nakayama rule we have, 
\begin{eqnarray*}
\chi^{((1,1),(2,2))}_{((2),(2,2))}&=&(-1)^{1+0.0}\chi^{(\emptyset,(2,2))}_{(\emptyset,(2,2))}+(-1)^{1+1.0}\chi^{((1,1),(1,1))}_{(\emptyset,(2,2))}+(-1)^{0+1.0}\chi^{((1,1),(2))}_{(\emptyset,(2,2))}\\
&=&-\chi^{(\emptyset,(2,2))}_{(\emptyset,(2,2))}-\chi^{((1,1),(1,1))}_{(\emptyset,(2,2))}+\chi^{((1,1),(2))}_{(\emptyset,(2,2))}.
\end{eqnarray*}
In order to obtain the exact value for $\chi^{((1,1),(2,2))}_{((2),(2,2))},$ we can apply once again the Murnaghan-Nakayama rule to obtain the values of $\chi^{(\emptyset,(2,2))}_{(\emptyset,(2,2))},$ $\chi^{((1,1),(1,1))}_{(\emptyset,(2,2))}$ and $\chi^{((1,1),(2))}_{(\emptyset,(2,2))}.$ All the recurrence steps are recorded in the following diagram
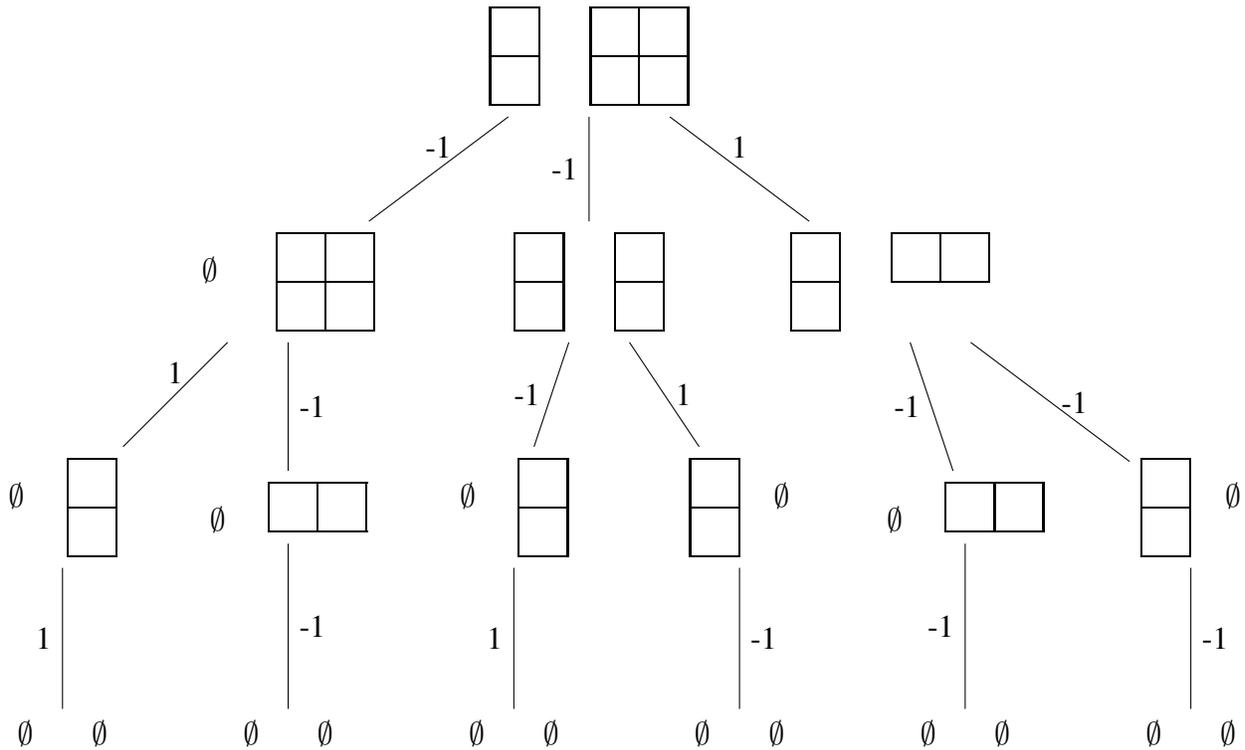
\[\begin{tikzpicture}
\node(1) at (0,0) {\ydiagram{1,1} \,\,\,\,\,\,\,\, \ydiagram{2,2}};
\node(2) at (-4,-3) {$\emptyset$ \,\,\,\,\,\,\,\, \ydiagram{2,2}};
\node(3) at (0,-3) {\ydiagram{1,1} \,\,\,\,\,\,\,\, \ydiagram{1,1}};
\node(4) at (4,-3) {\ydiagram{1,1} \,\,\,\,\,\,\,\, \ydiagram{2}};
\node(5) at (-7,-6) {$\emptyset$ \,\,\,\,\, \ydiagram{1,1}};
\node(6) at (-4,-6) {$\emptyset$ \,\,\,\,\, \ydiagram{2}};
\node(7) at (-1,-6) {$\emptyset$ \,\,\,\,\, \ydiagram{1,1}};
\node(8) at (2,-6) {\ydiagram{1,1} \,\,\,\,\, $\emptyset$};
\node(9) at (5,-6) {$\emptyset$ \,\,\,\,\, \ydiagram{2}};
\node(10) at (8,-6) {\ydiagram{1,1} \,\,\,\,\, $\emptyset$};
\node(11) at (-7,-9) {$\emptyset$ \,\,\,\,\,\,\,\, $\emptyset$};
\node(12) at (-4,-9) {$\emptyset$ \,\,\,\,\,\,\,\, $\emptyset$};
\node(13) at (-1,-9) {$\emptyset$ \,\,\,\,\,\,\,\, $\emptyset$};
\node(14) at (2,-9) {$\emptyset$ \,\,\,\,\,\,\,\, $\emptyset$};
\node(15) at (5,-9) {$\emptyset$ \,\,\,\,\,\,\,\, $\emptyset$};
\node(16) at (8,-9) {$\emptyset$ \,\,\,\,\,\,\,\, $\emptyset$};
\draw (1) -- node[above] {-1} (2) ;
\draw (1) -- node[left] {-1}  (3);
\draw (1) -- node[above] {1} (4);
\draw (2) -- node[above] {1} (5);
\draw (2) -- node[right] {-1} (6);
\draw (3) -- node[left] {-1} (7);
\draw (3) -- node[right] {1} (8);
\draw (4) -- node[left] {-1} (9);
\draw (4) -- node[right] {-1} (10);
\draw (5) -- node[left] {1} (11);
\draw (6) -- node[right] {-1} (12);
\draw (7) -- node[left] {1} (13);
\draw (8) -- node[right] {-1} (14);
\draw (9) -- node[left] {-1} (15);
\draw (10) -- node[right] {-1} (16);
\end{tikzpicture}\]
To obtain the value of $\chi^{((1,1),(2,2))}_{((2),(2,2))},$ one should sum up the numbers obtained by multiplying the integers along each path from the root to an and. Thus, $\chi^{((1,1),(2,2))}_{((2),(2,2))}=-1-1+1+1+1+1=2.$
\end{ex}

We define the \textit{content} of a box $b$ which is in position $(i,j)$ of the $t^{th}$ index of a bipartite partition of $n$ by $(-1)^t (j-i).$ For example, the content of the boxes $b_1$ and $b_2$ in the following bipartite partition of $7$

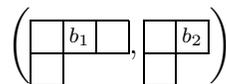
\[\left(
\tiny \ytableaushort{
{}{b_1}{},
{}
},
\ytableaushort{
{}{b_2},
{}
}
\right)
\]
are $(-1)^{0} (2-1)$ and $(-1)^{1} (2-1)$ respectively.  

A skew Young diagram $\lambda/\delta$ is a \textit{broken border strip} if it does not contain any $2\times 2$ block of cells. The reader should notice that the difference between a border strip and a broken border strip is that the first should be connected but the second may not be connected. In a broken border strip, we define a \textit{sharp corner} to be a box with a box below it and a box to its right, and a \textit{dull corner} to be a box with no box below it and no box to its right. For example, if $\lambda=(9,6,6,5,4,1,1)$ and $\delta=(6,6,4,4)$ then $\lambda/\delta$ is a broken border strip with three connected components, two sharp corners (marked with $s$) and five dull boxes (marked with $d$) as shown below 
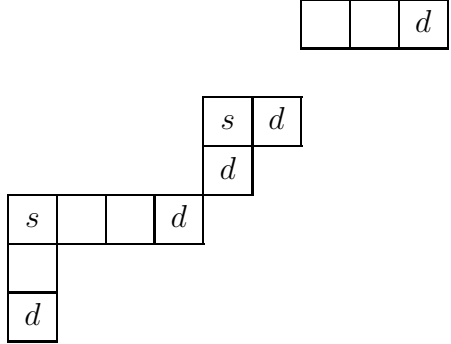
\[
\ytableaushort{
\none\none\none\none\none\none {}{}d,
\none\none\none\none\none\none,
\none\none\none\none s d,
\none\none\none\none d,
s {}{} d,
{},
d}
\]

We denote by $\SC(\lambda/\delta)$ the set of sharp corners of $\lambda/\delta$ and by $\DB(\lambda/\delta)$ its set of dull boxes. We say that $\blambda/\bdelta$ is a broken border strip if one of its constituents is a broken border strip and the other are $\emptyset.$ $\SC(\blambda/\bdelta)$ and $\DB(\blambda/\bdelta)$ are defined to be respectively the set of sharp corners and the set of dull boxes of its nonempty constituent.

\begin{prop}\label{Murnaghan_Nakayama_rule_hyperoctahedral}(Murnaghan-Nakayama rule for the generalized characters of Hyperoctahedral group) If $\bmu\nearrow \blambda$ and $\bdelta \nearrow \brho$ are two marked bipartitions of $n$ and if $j$ is the part of $\brho$ in which the marked box $\brho / \bdelta$ lies then
\[
\chi^{\bmu\nearrow \blambda}_{\bdelta \nearrow \brho}=\sum_{\bnu\subseteq \bmu \atop{ \bnu\vdash n-j}}c_{\bmu,\blambda,\bdelta,\brho,\bnu} \chi^{\bnu}_{\brho- (j)}
\]
where 
\[
c_{\bmu,\blambda,\bdelta,\brho,\bnu}=
\begin{cases}
(-1)^{\height(\blambda/\bnu)+P(\blambda/\bnu).P(j)} \frac{\displaystyle \prod_{s\in \SC(\blambda/\bnu)} \left[ c(\blambda/\bmu)-c(s)\right]}{\displaystyle \prod_{d\in \DB(\blambda/\bnu), \atop{d\neq \blambda/\bmu}} \left[ c(\blambda/\bmu)-c(d)\right]} & \text{ if $\blambda/\bnu$ is a broken border strip,}\\
0 & \text{ otherwise.}
\end{cases}
\]
\end{prop}

\begin{proof} The proof is similar to the one given in \cite{strahov2007generalized} for the Murnaghan-Nakayama rule of the geberalized characters of the symmetric group. By \cite[Proposition 4.3.1]{strahov2007generalized}, we have
\begin{equation}
\chi^{\bmu\nearrow \blambda}_{\bdelta \nearrow \brho}=\sum_{S}\bigtriangleup(\bmu^{(1)})\bigtriangleup(\bmu^{(2)}/\bmu^{(1)})\cdots \bigtriangleup(\bmu^{(k-1)}/\bmu^{(k-2)})\bigtriangleup(\bmu/\bmu^{(k-1)};\blambda/\bmu^{(k-1)}),
\end{equation}
where the sum is over all sequences $\emptyset=\bmu^{(0)}\subseteq \bmu^{(1)}\subseteq\cdots\subseteq\bmu^{(k)}=\bmu$ of $k$-partite rim hook tableau of shape $\bmu$ such that $l(rh_i):=|\bmu^{(i)}/\bmu^{(i-1)}|=\rho^i\neq j,$ for all $1\leq i\leq k-1;$ and $|\bmu^{(k)}/\bmu^{(k-1)}|= j-1,$ with
\begin{equation}
\bigtriangleup(\bmu^{(i)}/\bmu^{(i-1)}):=\sum_{T\in SYT(\bmu^{(i)}/\bmu^{(i-1)})}\frac{1}{(c_T(2)-c_T(1))\cdots (c_T(\rho^i)-c_T(\rho^i-1))}
\end{equation}
and
\[
\bigtriangleup(\bmu/\bmu^{(k-1)};\blambda/\bmu^{(k-1)}):=
\]
\begin{equation}
\sum_{T\in SYT(\bmu/\bmu^{(k-1)})}\frac{1}{(c_T(2)-c_T(1))\cdots (c_T(j-1)-c_T(j-2))(c_T(\blambda/\bmu)-c_T(j-1))}.
\end{equation}
To obtain the Murnaghan-Nakayama rule, we can compute the terms $\bigtriangleup(\bmu^{(i)}/\bmu^{(i-1)})$ by using \cite[Theorem 4.4.1]{strahov2007generalized}. In addition, the remaining term $\bigtriangleup(\bmu/\bmu^{(k-1)};\blambda/\bmu^{(k-1)})$ can be computed using Theorem 4.5.1 and Proposition 4.6.2 of \cite{strahov2007generalized}. 

\end{proof}

\begin{ex} Suppose $n=2,$ $\bmu=((1),\emptyset),$ $\blambda=((1,1),\emptyset),$ $\bdelta=((1),\emptyset)$ and $\brho=((2),\emptyset).$ Then, $j=2$ and $(\emptyset,\emptyset)$ is the only $\bnu$ to contribute to the summation in the above rule. Therefore, $\chi^{\bmu\nearrow \blambda}_{\bdelta \nearrow \brho}=c_{\bmu,\blambda,\bdelta,\brho,(\emptyset,\emptyset)}=(-1)^{1+0.0}.\frac{1}{1}=-1$ as both $\SC(\blambda/\bnu)$ and $\lbrace d\in \DB(\blambda/\bnu), d\neq \blambda/\bmu\rbrace$ are emptysets. The reader is invited to check that this result agrees with Table \ref{H1 gen char tab} of the $\mathcal{H}_1$-generalized characters of $\mathcal{H}_2$
\end{ex}

\begin{ex} Suppose $n=3,$ $\bmu=((2),\emptyset),$ $\blambda=((2,1),\emptyset),$ $\bdelta=(\emptyset,(2))$ and $\brho=(\emptyset,(3)).$ Then, $j=3$ and $(\emptyset,\emptyset)$ is the only $\bnu$ to contribute to the summation in the above rule. Therefore, $\chi^{\bmu\nearrow \blambda}_{\bdelta \nearrow \brho}=c_{\bmu,\blambda,\bdelta,\brho,(\emptyset,\emptyset)}=(-1)^{1+0.1}.\frac{-1-0}{-1-1}=-\frac{1}{2}$ as $\SC(\blambda/\bnu)$ contains only the below red box, $\lbrace d\in \DB(\blambda/\bnu), d\neq \blambda/\bmu\rbrace$ consists of the below blue box and $\blambda/\bmu$ is the below yellow box

\[
\Big( \ytableausetup{nosmalltableaux}
\begin{ytableau}
*(red) 0& *(blue) 1  \\
*(yellow) -1
\end{ytableau},\emptyset \Big)
\]

This result can be verified using the definition in Equation (\ref{gen_char_hyp}) of the generalized characters along with the character tables of $\mathcal{H}_2$ and $\mathcal{H}_3$ that can be found in \cite{islami2021symmetric}. Indeed, if $x=(135246)$ is the permutation of $\mathcal{H}_3$ with marked-type $(\emptyset,(3^*)),$ then by definition
\[
\chi^{\bmu\nearrow \blambda}_{\bdelta \nearrow \brho}=\frac{\chi^{\bmu}(1)}{2^{2}(2)!}\sum_{h\in \mathbb{Z}_2 \wr \mathcal{S}_{2}}\chi^{\blambda}(xh)\chi^{\bmu}(h)=
\frac{1}{2^{2}(2)!}\big(\chi^{\blambda}(x)\chi^{\bmu}(1)+\chi^{\blambda}((135)(246))\chi^{\bmu}((12)(3)(4))
\]
\[+\chi^{\blambda}((146)(235))\chi^{\bmu}((1)(2)(34))+\chi^{\blambda}((146235))\chi^{\bmu}((12)(34))+\chi^{\blambda}((1)(2)(3546))\chi^{\bmu}((13)(24))
\]
\[+\chi^{\blambda}((12)(35)(46))\chi^{\bmu}((14)(23))+\chi^{\blambda}((12)(3546))\chi^{\bmu}((1324))+
\chi^{\blambda}((1)(2)(35)(46))\chi^{\bmu}((1423))\big)
\]
\[
=\frac{1}{8}\big((-1).(1)+(-1).(1)+(-1).(1)+(-1).(1)+(0).(1)+(0).(1)+(0).(1)+(0).(1) \big)=-\frac{1}{2}.
\] 

\end{ex}

\begin{ex} Suppose $n=3,$ $\bmu=((1),(1)),$ $\blambda=((1),(2)),$ $\bdelta=((1),(1))$ and $\brho=((1),(2)).$ Then, $j=2$ and both bipartitions $((1),\emptyset)$ and $(\emptyset,(1))$ contribute to the summation in the above rule. Therefore, 
$$\chi^{\bmu\nearrow \blambda}_{\bdelta \nearrow \brho}=c_{\bmu,\blambda,\bdelta,\brho,((1),\emptyset)}\chi^{((1),\emptyset)}_{((1),\emptyset)}+c_{\bmu,\blambda,\bdelta,\brho,(\emptyset,(1))}\chi^{(\emptyset,(1))}_{((1),\emptyset)}.$$ 
But $c_{\bmu,\blambda,\bdelta,\brho,(\emptyset,(1))}=0$ since $\blambda/(\emptyset,(1))=((1),(1))$ is not a broken border strip, $c_{\bmu,\blambda,\bdelta,\brho,((1),\emptyset)}=-1^{0+1.1}.\frac{1}{1}$ and $\chi^{((1),\emptyset)}_{((1),\emptyset)}=1.$ Thus, $\chi^{\bmu\nearrow \blambda}_{\bdelta \nearrow \brho}=-1.$
\end{ex}

By using Proposition \ref{Murnaghan_Nakayama_rule_hyperoctahedral}, we show below the complete rows corresponding to the marked bipartitions $((2,1^*),\emptyset)$ and $((1),(2^*))$ in the table of the $\mathcal{H}_2$-generalized characters of $\mathcal{H}_3.$

\bigskip

\begin{tabular}{p{2cm}p{2cm}p{2cm}p{2cm}p{2cm}p{2cm}p{2cm}}
&$((1,1,1^*),\emptyset)$&$((2^*,1),\emptyset)$&$((2,1^*),\emptyset)$&$((3^*),\emptyset)$&$((1,1^*),(1))$&$((1,1),(1^*))$\\
 \hline
\text{Order}    &   1 & 4 & 2 & 8 & 2 & 1 \\
\hline
$\chi^{((2,1^*),\emptyset)}$    &   1 & -1/2 & 1  & -1/2 & 1 & 1 \\
\hline
$\chi^{((1),(2^*))}$    &   2 & 1 & 0 & 0 & 0 & -2 \\
\hline
\end{tabular}

\bigskip

\begin{tabular}{p{2cm}p{2cm}p{2cm}p{2cm}p{2cm}p{2cm}p{2cm}}
&$((2^*),(1))$&$((2),(1^*))$&$((1^*),(1,1))$&$((1),(1,1^*))$&$((1^*),(2))$&$((1),(2^*))$\\
 \hline
\text{Order}    &  4 & 2 & 1 & 2 & 2 & 4 \\
\hline
$\chi^{((2,1^*),\emptyset)}$    &  -1/2 & 1 & 1 & 1 & 1 & -1/2 \\
\hline
$\chi^{((1),(2^*))}$    &  1 & 0 & -2 & 0 & 0 & -1 \\
\hline
\end{tabular}

\bigskip

\begin{tabular}{p{2cm}p{2cm}p{2cm}p{2cm}p{2cm}}
&$(\emptyset,(1,1,1^*))$&$(\emptyset,(2^*,1))$&$(\emptyset,(2,1^*))$&$(\emptyset,(3^*))$ \\
 \hline
\text{Order}    &   1 & 4 & 2 & 8\\
\hline
$\chi^{((2,1^*),\emptyset)}$    &  1 & -1/2 & 1 & -1/2\\
\hline
$\chi^{((1),(2^*))}$    &  2 & -1 & 0 & 0\\
\hline
\end{tabular}

\begin{cor}\label{Main_cor} If $\bdelta=(\emptyset,(n-1)),$ $\brho=(\emptyset,(n))$ and $a+b+1=n,$ then 
\[
\chi^{\bmu\nearrow \blambda}_{\bdelta \nearrow \brho}=
\begin{cases}
(-1)^b\frac{a}{a+b} & \text{if }  \bmu=((a,1^b),\emptyset) \text{ and } \blambda=((a+1,1^b),\emptyset)\\
(-1)^b\frac{b}{a+b} & \text{if } \bmu=((a+1,1^{b-1}),\emptyset) \text{ and } \blambda=((a+1,1^b),\emptyset)\\
(-1)^{b+1}\frac{a}{a+b} & \text{if } \bmu=(\emptyset,(a,1^b)) \text{ and } \blambda=(\emptyset,(a+1,1^b))\\
(-1)^{b+1}\frac{b}{a+b} & \text{if } \bmu=(\emptyset,(a+1,1^{b-1})) \text{ and } \blambda=(\emptyset,(a+1,1^b))\\
0 & \text{ otherwise.}
\end{cases}
\]
\end{cor}
\begin{proof}
Direct consequence of Proposition \ref{Murnaghan_Nakayama_rule_hyperoctahedral}.
\end{proof}

\begin{ex}
The nonzero entries in the column corresponding to the marked bipartition $(\emptyset,(3^*))$ in the table of the $\mathcal{H}_2$-generalized characters of $\mathcal{H}_3$ are given below

\bigskip

\begin{tabular}{p{2cm}p{2cm}}
&$(\emptyset,(3^*))$ \\
 \hline
\text{Order}     & 8\\
\hline
$\chi^{((1,1,1^*),\emptyset)}$    &    1\\
\hline
$\chi^{((2^*,1),\emptyset)}$    &     -1/2\\
\hline
$\chi^{((2,1^*),\emptyset)}$    &    -1/2\\
\hline
$\chi^{((3^*),\emptyset)}$    &     1\\
\hline
$\chi^{(\emptyset,(1,1,1^*))}$    &     -1\\
\hline
$\chi^{(\emptyset,(2^*,1))}$    &     1/2\\
\hline
$\chi^{(\emptyset,(2,1^*))}$    &     1/2\\
\hline
$\chi^{(\emptyset,(3^*))}$    &   -1
\end{tabular}
\end{ex}

\subsection*{Availability of data and material.} Data sharing not applicable to this article as no datasets were generated or analysed during the current study.

\subsection*{Code availability.} Not applicable.

\bibliographystyle{plain}
\bibliography{biblio}

\begin{thebibliography}{1}

\bibitem{brender1976spherical}
M.~Brender.
\newblock Spherical functions on the symmetric groups.
\newblock {\em Journal of Algebra}, 42(2):302--314, 1976.

\bibitem{ceccherini2010representation}
T.~Ceccherini-Silberstein, F.~Scarabotti, and F.~Tolli.
\newblock {\em Representation theory of the symmetric groups: the
  Okounkov-Vershik approach, character formulas, and partition algebras},
  volume 121.
\newblock Cambridge University Press, 2010.

\bibitem{islami2021symmetric}
A.~Islami.
\newblock Symmetric functions as characters of hyperoctahedral group.
\newblock Ph.D. thesis, York University, 2020.

\bibitem{McDo}
I.G. Macdonald.
\newblock {\em Symmetric functions and Hall polynomials}.
\newblock Oxford Univ. Press, second edition, 1995.

\bibitem{sagan2001symmetric}
B.~E. Sagan.
\newblock {\em The symmetric group: representations, combinatorial algorithms,
  and symmetric functions}, volume 203.
\newblock Springer, 2001.

\bibitem{stein2017littlewood}
I.~Stein.
\newblock The littlewood-richardson rule for wreath products with symmetric
  groups and the quiver of the category {F}$\wr$ {FI}$_n$.
\newblock {\em Communications in Algebra}, 45(5):2105--2126, 2017.

\bibitem{strahov2007generalized}
E.~Strahov.
\newblock Generalized characters of the symmetric group.
\newblock {\em Advances in Mathematics}, 212(1):109--142, 2007.

\bibitem{toutejc}
O.~Tout.
\newblock Structure coefficients of the hecke algebra of $({S}_{2n}, {B}_n)$.
\newblock {\em The Electronic Journal of Combinatorics}, 21(4):P4--35, 2014.

\bibitem{tout2021symmetric}
O.~Tout.
\newblock On the symmetric gelfand pair $ (\mathcal{H}_n \times
  \mathcal{H}_{n-1}, \diag (\mathcal{H}_{n-1}))$.
\newblock {\em Journal of Algebra and Its Applications}, 20(05):2150085, 2021.

\end{thebibliography}

\end{document}